\documentclass{amsart}
\usepackage{amsmath,amsthm,amscd,amssymb}
\usepackage{mathrsfs}
\usepackage{graphicx}
\usepackage{enumerate}
\usepackage{accents}
\usepackage{bm}

\usepackage{xcolor}
\usepackage{color}

\usepackage{url}
\usepackage[
]{hyperref}

\usepackage[mathlines]{lineno}

\makeatletter
\@namedef{subjclassname@2020}{\textup{2020} Mathematics Subject Classification}
\makeatother

\newtheorem{mthm}{Theorem}

\theoremstyle{plain}
\newtheorem{thm}{Theorem}[section]
\newtheorem{cor}[thm]{Corollary}

\newtheorem{lem}[thm]{Lemma}
\newtheorem{clm}[thm]{Claim}
\newtheorem{prop}[thm]{Proposition}

\newtheorem{problem}[thm]{Problem}

\theoremstyle{definition}
\newtheorem{dfn}[thm]{Definition}

\theoremstyle{remark}
\newtheorem{remark}[thm]{Remark}
\newtheorem*{pf}{Proof}
\numberwithin{equation}{section}
\numberwithin{figure}{section}

\newcommand{\Li}{\mathcal{L}}

\newcommand{\la}{\Lambda}
\newcommand{\A}{A}

\newcommand{\D}{D}
\newcommand{\ubar}{\underline}

\def\smfd{W^{\mathrm{s}}}
\def\umfd{W^{\mathrm{u}}}
\def\lsmfd{W^{\mathrm{s}}_{\mathrm{loc}}}

\def\lumfd{W^{\mathrm{u}}_{\mathrm{loc}}}

\def\Diff{\mathrm{Diff}}

\title[Takens' Last Problem and strong pluripotency]
{
Takens' Last Problem and strong pluripotency
}

\author{Shin Kiriki}

\address[Shin Kiriki]{Department of Mathematics, Tokai 
University, 4-1-1 Kitakaname, Hiratuka, Kanagawa, 259-1292, JAPAN}
\email{kiriki@tokai.ac.jp}

\author{Xiaolong Li$^{\ast}$}\thanks{$\ast$\ Author to whom any correspondence should be addressed}

\address[Xiaolong Li]{School of Mathematics and Statistics, 
Huazhong University of Science and Technology, Luoyu Road 1037, Wuhan, 430074, CHINA}
\email{lixl@hust.edu.cn}

\author{Yushi Nakano}

\address[Yushi Nakano]{
Department of Mathematics, Hokkaido University, Kita 10, Nishi 8, Kita-Ku, Sapporo, Hokkaido, 060-0810, 
JAPAN}
\email{yushi.nakano@math.sci.hokudai.ac.jp}

\author{Teruhiko Soma}

\address[Teruhiko Soma]{Department of Mathematical Sciences, 
Tokyo Metropolitan University, 1-1 Minami-Ohsawa, Hachioji, Tokyo, 192-0397, JAPAN}
\email{tsoma@tmu.ac.jp}

\author{Edson Vargas 
}
\address[Edson Vargas]{Departamento de Matematica, 
IME-USP, S\~ao Paulo, BRAZIL}
	\email{vargas@ime.usp.br}


\begin{document}

\begin{abstract}
We consider the concept of strong pluripotency of dynamical systems for a hyperbolic invariant set, 
as introduced in \cite{KNS}. 
To the best of our knowledge, 
for the whole hyperbolic invariant set,
the existence of robust strongly pluripotent dynamical systems has not been proven in previous studies. 
In fact, there is an example of strongly pluripotent dynamical systems in \cite{CV01},
but its robustness has not been proven. On the other hand, 
robust strongly pluripotent dynamical systems for some proper subsets of hyperbolic sets 
had been found in \cite{KS17, KNS}.
In this paper, we provide a combinatorial way to recognize strongly pluripotent diffeomorphisms in a Newhouse domain and prove that they are $C^r$-robust, $2\leq r< \infty$.
More precisely, we prove that there is a two-dimensional diffeomorphism with a wild Smale horseshoe which has a  
 $C^r$ neighborhood $\mathcal{U}_0$ where all elements are strongly pluripotent for the whole Smale horseshoe.
Moreover, it follows from the result that any property, 
such as 
having a non-trivial physical measure supported by the Smale horseshoe 
or
having historic behavior,
is $C^r$-persistent relative to a dense subset of $\mathcal{U}_0$.
\end{abstract}

\maketitle

\section{Introduction}

In this paper, we consider open subsets of the space $\Diff^r(M)$ of $C^r$ diffeomorphisms endowed with the 
$C^r$ topology,  where $M$ is a compact Riemannian surface without boundary.
For a large subset of $\Diff^r(M)$, those satisfying the Axiom A for example, 
the topological and statistical behavior of almost every (in the Lebesgue sense) forward orbits 
agree, that is, they are governed by well-understood measures supported on the topological  attractors, see \cite{Sinai72,Bowen, R03}.

On the other hand, following Ruelle \cite{R02}, among the possible 
statistical behaviors, there are 
points $x \in M$ or its forward orbits which have {\em historic behavior}, that is, 
points $x$ such that the sequence of 
empirical  measures 
\begin{equation}\label{ep}
\delta^{n}_{x,f}=
\frac{1}{n}\sum_{i=0}^{n-1}\delta_{f^i(x)},
\end{equation}
 where $\delta_{f^i(x)}$ is the Dirac measure at $f^i(x),$	
does not converge in the weak$^{\ast}$ topology, when $n$ goes to 
infinity. 
The set of points with historic behavior of an Axiom A diffeomorphism has zero Lebesgue measure
but, in some cases, it is a residual set.
For example, it follows from \cite{T08} that this is the case of the well-known solenoid on the 
three-dimensional solid torus. 
Nevertheless, 
it is natural to ask 
the existence and abundance of diffeomorphisms whose set of initial points with historic behavior has positive Lebesgue measure.
Indeed, Takens' Last Problem \cite{T08} is \emph{whether there is a persistent class of dynamical systems
such that the set of initial points with historic behavior has positive Lebesgue measure}.
Here, we recall the concepts of persistence and robustness to avoid confusion. 
Let $\mathcal{C}$ be a non-empty subset of
$\mathrm{Diff}^r(M)$, which is called a \emph{class}. 
We say that a property  $\mathscr{A}$ is $C^r$-\emph{persistent} relative to $\mathcal{C}$ if every $f\in \mathcal{C}$ has the property $\mathscr{A}$. See \cite[Section 11]{PR83} and \cite[Section 3]{T08}.
 Such a property is called $C^r$-\emph{robust},  particularly when $\mathcal{C}$ is an open set.

As an answer to Takens' Last Problem in dimension two,
Kiriki and Soma  \cite{KS17} proved that the property of having a wandering domain with historic behavior is 
$C^r$-persistent relative to a dense subset of every Newhouse domain. 
This was extended in several directions \cite{LR17, B22, BB23}.
Here we go further and prove in Theorem \ref{mainthm} that there exist 
a diffeomorphism $F_{0}\in \Diff^{r}(M)$ in  a Newhouse domain and its 
$C^{r}$ neighborhood $\mathcal{U}_0$ all elements of which are strongly pluripotent for a Smale horseshoe.
The concept of strong pluripotency, 
Definition~\ref{dfn1}, was borrowed 
from  \cite{KNS} where it appeared for the first time.
Roughly speaking,
it implies that 
any orbit starting from Smale horseshoe for a diffeomorphism in $\mathcal{U}_{0}$, 
whose statistical behavior is arbitrarily prescribed in a combinatorial manner, 
can be realized by some nearby diffeomorphism and  a set 
of points with positive Lebesgue measure.



As another result in Theorem \ref{HD},  we distinguish two dense classes with completely different statistical properties of $\mathcal{U}_0$.
One is the class of diffeomorphisms
which have a non-trivial physical measure supported on some saddle orbit in a Smale horseshoe.
The other is the class of diffeomorphisms $g$ which have 
a wandering domain $D$ 
such that, for every $x\in D$, 
the set
of weak$^{\ast}$ accumulation of the empirical measures
$\delta^n_{x, g}$ contains at least two different measures 
supported on the Smale horseshoe.

\section{Basic concepts and main results}
Throughout this paper, let $r$ be a fixed integer with $2\leq r<\infty$ except the arguments on $C^{1+\alpha}$ topology in Subsection \ref{ssH}, 
%
%
$M$ a  compact $C^r$ Riemannian surface  
without boundary and $\Diff^r(M)$
the set of all $C^r$ diffeomorphisms of $M$ endowed with 
the $C^r$ topology. 

\subsection{Wandering domain and pluripotency}
Let us recall several topological concepts of dynamical systems. 
We say that $\A \subset M$ is an {\em attractor} for
$f \in \Diff^r(M),$ if $\A$ is a compact $f$-invariant  transitive
set  and its basin of attraction
$$
B_f(\A) = \{x \in M : f^n (x) \to  \A  \textrm{ as }   n \to  +\infty \}
$$
contains a neighborhood of $\A$. Moreover
we say that  $\A$ is a \emph{weak attractor}  for $f$
if it satisfies the following conditions.
\begin{itemize}
\item $\A$ is a non-wandering and 
dynamically connected invariant set
(i.e.\ it is not the union of two non-trivial closed disjoint invariant sets),
\item $B_f(\A)$ contains an open set $C$ which has only finitely many 
connected components and such that the closure $c\ell(C)$ contains $\A$.
\end{itemize}

\begin{dfn}\label{dfnWD}
A non-empty connected open set $D\subset M$ is called a \emph{wandering domain} for $f$ if $f^n(D)$ ($n=0,1,\dots$) are pairwise disjoint. Furthermore, 
\begin{enumerate}[(i)]
\item A wandering domain $\D$
is called {\em contracting} if the diameter
of $f^n(\D) $ goes down to zero as $n$ goes to $+\infty$.
\item\label{dfnNWD} A wandering domain $\D$ is called \emph{non-trivial} if $\D$ is not contained in the basin of attraction of a weak attractor.
\end{enumerate}  
\end{dfn}
Note that the definition \eqref{dfnNWD} for non-triviality of the wandering domain in Definition \ref{dfnWD} is stronger than the condition of \cite{dMvS}. In fact, for this stronger condition, 
 a wandering domain of Denjoy's example on $S^{1}$ is no longer non-trivial 
 since the basin of a weak attractor is the whole  $S^{1}$.
But the stronger condition is more effective than that of \cite{dMvS}
 in eliminating several trivial examples of higher dimensions, 
e.g. Bowen eye or contracting saddle-node, see \cite{CV01}.

Next, 
we recall the \emph{first Wasserstein metric} $d_W$ for any Borel probability measures $\mu$ and $\nu$ on $M$ 
defined as
\[
d_W(\mu ,\nu )= \sup_{\varphi \in \Li} \left\vert \int _M \varphi \, 
d\mu - \int _M \varphi \, d\nu \right\vert,\]
where $\Li$ is the set of Lipschitz  functions $\varphi : M \to [-1, 1]$ with 
Lipschitz constants bounded by $1$. 
We now formulate the concepts of pluripotency and strong 
pluripotency which appeared for the first time in \cite{KNS}.

\begin{dfn}[pluripotency]\label{dfn1}
Let $ \la_f $ be a uniformly hyperbolic compact invariant set (for simplicity, hyperbolic set) 
for $f \in \Diff^{r}(M).$ 
\begin{itemize}
\item $f$  is said to be \emph{pluripotent} for $\la_f$ if, for any $x\in \la_f,$ 
there exists $g \in \Diff^r(M),$ arbitrarily $C^r$-close to $f,$ which
has a set of positive Lebesgue measure $\D_g$  such that, 
for any $y\in \D_g$, 
\begin{equation}\label{defpl}
\lim _{n\to \infty} d_W( \delta_{y,g}^n, \delta_{x_g,g}^n ) =0,
\end{equation}
where 
$\delta_{y,g}^n$ and $\delta_{x_g,g}^n$ are empirical measures given as
\eqref{ep} and $x_g\in \la_g$ is the continuation of $x\in \la_f$. 
\item
$f$  is said to be \emph{strongly pluripotent} for $\la_f$,   if the above 
\eqref{defpl} 
is replaced by 
\begin{equation}\label{defspl}
\lim _{n\to \infty} 
\frac{1}{n}\sum_{i=0}^{n-1}
\sup_{y\in \D_{g}}\mathrm{dist}(g^i(y),g^i(x_{g}))=0.
\end{equation}
\end{itemize}
\end{dfn}

\begin{remark}
We note something important about the above definition:
\begin{enumerate}
\item 
In the results of this paper, 
the set corresponding to $\D_{g}$
in Definition \ref{dfn1} is provided as a non-trivial wandering domain.
\item
 It can be shown that \eqref{defspl} implies \eqref{defpl} while the converse is not true in general, see \cite{KNS}.
 \item 
We can generalize Definition \ref{dfn1} to 
a subset of the hyperbolic invariant set rather than the entire hyperbolic invariant set.
Compare Definition \ref{dfn1} with its generalized version in \cite{KNS}.

\end{enumerate}
\end{remark}

Pluripotency is a term widely used in physiology and related fields to refer to the ability of a system to move from an undifferentiated state to various states determined by its internal dynamics.
In fact, Yamanaka was awarded the Nobel Prize for developing a technique to reprogram somatic cells, introducing pluripotent stem cells (iPSCs) by a small change of genes \cite{Yam12}.
The above definition is an abstraction of the concept of pluripotency from a dynamical systems perspective.

\subsection{Robust strongly pluripotency}
In this paper, we prove the existence of an open set of strongly pluripotent diffeomorphisms 
$f$
 for a \emph{wild Smale horseshoe}: 
a uniformly hyperbolic invariant set $\Lambda$ 
such that the restriction $f\vert_{\Lambda}$
is topologically conjugated to the shift map on the full two-sided two-symbol space 
and besides which has a homoclinic tangency. 
Note that 
the property of having a homoclinic tangency is 
$C^{r}$-persistent relative to an open set, namely $C^{r}$-robust, in $\Diff^{r}(M)$. 
Such an open set is called a \emph{Newhouse domain}.
Newhouse domains are shown to be non-empty by 
Newhouse for two-dimensional  $C^{2}$ diffeomorphisms \cite{N70, N74, N79},
by Crovisier et al.\ for two-dimensional $C^{1+\alpha}$ diffeomorphisms \cite{C13} (see Subsection \ref{ssH} below), 
and 
by  Bonatti--D\'{\i}az  for three or higher dimensional $C^{1}$ diffeomorphisms \cite{BD96}.
See \cite{PT93, BDV05} for a comprehensive explanation.

The starting point for pluripotency is based on ideas used in  \cite{CV01} to prove the existence of wandering domains with different ergodic properties. 
These ideas were also adopted in \cite{KS17} to give an affirmative solution to Takens' Last Problem in two-dimensional diffeomorphisms.
In \cite{KNS23}, 
similar ideas were extended to three-dimensional diffeomorphisms with a wild blender-horseshoe 
and we studied statistical dynamics of the contracting wandering domain.
In \cite{KNS}, the pluripotency and strong pluripotency are formulated for the first time 
and these properties 
are shown 
to be $C^{2}$-robust for a diffeomorphism with a wild blender-horseshoe. 
Note, however, that the pluripotency studied in \cite{KNS} is limited to some proper subsets of blender-horseshoe. 
In other words, 
\emph{it remains unknown whether there exists an open set of strongly pluripotent diffeomorphisms 
for the whole part of a basic set such as a horseshoe even in two-dimension}.
To this problem, we give the next result for two-dimensional diffeomorphisms.
 
\begin{mthm}\label{mainthm} 
There are an element $F_{0}$ of $\Diff^r(M)$ having a wild Smale horseshoe $\la_{F_{0}}$ and a 
 $C^{r}$ neighborhood  $\mathcal{U}_0$ of $F_{0}$ 
such that every diffeomorphism $f \in \mathcal{U}_0$  is  strongly
pluripotent for the continuation $\la_{f}$ of $\la_{F_{0}}$. 
\end{mthm}

The positive Lebesgue measure set in the proof of Theorem \ref{mainthm} corresponding to $D_{g}$ in Definition \ref{dfn1} is given as 
a non-trivial wandering domain for some $g$ arbitrarily $C^{r}$ close to $f$.
See the subsequent sections for details.
Note that, by the result in \cite{KS17}, we might obtain a similar conclusion in some open sets in $\mathcal{U}_0$ close to $F_{0}$, 
but there is no guarantee in such a way that the conclusion is correct in the whole   $\mathcal{U}_0$.

Let $\mathcal{P}_f(\Lambda_f)$ be the space of all  $f$-invariant probability measures supported on $\Lambda_f$, equipped with the first Wasserstein metric. The limit set of $(\delta_{x,f}^n)_{n\ge 0}$ is denoted by $\omega((\delta_{x,f}^n)_{n\ge 0})$. 
The following corollary is obtained from Theorem A together with \cite[Theorem 4]{S74}. See also \cite[Subsection 1.2]{KNS22} for related topics.
\begin{cor}\label{cor:B}
For any $f \in \mathcal{U}_0$, there exist an element $g \in \mathcal{U}_0$ arbitrarily $C^r$-close to $f$ and a non-wandering domain $D_g$ of $g$ such that, for any $x \in D_g$, we have
\[
\omega((\delta_{x,g}^n)_{n\ge 0}) = \mathcal{P}_g(\Lambda_g).
\]
\end{cor}

\begin{remark}
In \cite[Theorem B]{BB23}, Berger and Biebler proved that, for any element $f$ of the dissipative Newhouse domain $\mathcal{N}^r$ in $\mbox{Diff}^r(M)$, there exist $g\in\mathcal{N}^r$ arbitrarily $C^r$-close to $f$, a non-wandering domain $D_g$ of $g$, a constant $t\in (0,1)$ and $\mu \in \mathcal{P}_g(\Lambda_g)$ such that for any $x \in D_g$, the limit set $\omega((\delta_{x,f}^n)_{n\ge 0})$ contains the proper subset $
\{t\mu + (1 - t)\nu \mid \nu \in \mathcal{P}_g(\Lambda_g)\} $
of $\mathcal{P}_g(\Lambda_g)$. However, since their theorem has not shown the equality of the corollary, we are not convinced that $f$ is strongly pluripotent for $\Lambda_f$.
\end{remark}

\subsection{Persistent properties in $\mathcal{U}_0$}
Next, we state that any strongly pluripotent diffeomorphism in  $\mathcal{U}_0$ can be approximated by two classes with completely different statistical properties defined as follows.

The  first property is the existence of  \emph{non-trivial physical measure} $\mu$ 
satisfying the following conditions:
\begin{itemize}
\item for  $g\in \mathcal{U}_0$, there exists $x\in M$ such that 
$(\delta^{n}_{x,g})_{n\geq 0}$ converges to an invariant Borel measure $\mu$ whose support  $\textrm{supp}(\mu)$ 
is not an attractor,
\item the basin $B_g(\mu)=
\{ x \in M : \delta^n_{x,g} 
\overset{\mathrm{weak}^*}\longrightarrow \mu\ \textrm{as}\ 
n \to \infty \}$ of $\mu$ has positive Lebesgue measure.
\end{itemize}
Moreover, we say 
that a non-trivial physical measure $\mu$ is \emph{Dirac} if the support of $\mu$ is a periodic orbit of saddle type.
We call a point $x$ \emph{Birkhoff regular} for $g$ 
if 
$(\delta^{n}_{x,g})_{n\geq 0}$ converges to an invariant Borel measure.
For a Birkhoff regular point $q\in \Lambda_{F_{0}}$, 
we say that $g\in \mathcal{U}_0$ satisfies the property $\mathscr{D}_{q}$ if $g$ has a non-trivial physical measure whose support is the closure of the orbit of the continuation $q_{g}$ of $q$.


The second one is the existence of \emph{historic behavior}  which already appeared in the previous section: 
\begin{itemize}
\item
for  $g\in \mathcal{U}_0$ there exists $x\in M$ such that 
the sequence $(\delta^{n}_{x,g})_{n\geq 0}$ of empirical measures given as 
\eqref{ep}
does not converge.
\end{itemize}
We say that $g\in \mathcal{U}_0$ satisfies the property  $\mathscr{H}$ if $g$ has a non-trivial wandering domain such that 
the $g$-forward orbit of each point in the domain has historic behavior.


\begin{mthm}\label{HD}
 Suppose that $\mathcal{U}_0$ is the $C^{r}$ neighborhood of $F_{0}$ in Theorem \ref{mainthm}. Then there exist a dense class $\mathcal{H}$ of $\mathcal{U}_0$ and, for any Birkhoff regular  point $q\in \Lambda_{F_{0}}$, another dense class $\mathcal{D}_{q}$ of $\mathcal{U}_0$  satisfying the following conditions:
\begin{enumerate}[\rm (1)]
\item \label{HD2} $\mathscr{D}_{q}$ is $C^{r}$-persistent relative to $\mathcal{D}_{q}$.
\item \label{HD1} $\mathscr{H}$ is $C^{r}$-persistent relative to $\mathcal{H}$.
\end{enumerate}
\end{mthm}

\begin{remark}
The result of Theorem \ref{HD} \eqref{HD1} is an affirmative answer to Takens' Last Problem 
which could not be obtained from that of \cite{KS17}. 
Indeed, though the result of \cite{KS17} provides a locally dense subset of an
open set arbitrarily close to $F_{0}$, 
it does not guarantee that it is dense in a neighborhood of $F_{0}$
\end{remark}

\subsection{The $C^{1+\alpha}$ case}\label{ssH}
Theorems \ref{mainthm} and \ref{HD}  would also hold for every real number $r=1+\alpha$ with $0<\alpha<1$.
It is well-known that  homoclinic tangencies for two-dimensional diffeomorphisms 
exist $C^{r}$-robustly  if $r\geq 2$ \cite{N79} 
but does not if $r=1$ \cite{M11}. 
On the other hand, it was not publicly known whether two-dimensional diffeomorphisms have $C^{r}$-robust homoclinic tangency when $1<r<2$. 
But two years later after \cite{M11}, 
Crovisier and Gourmelon gave a positive answer to the problem
and recently provided it in the lecture note \cite[Remark 1]{C13}.
The most important part of their proof is that it presents a new way to evaluate overlappings of stable and unstable laminations of a horseshoe of a $C^{1+\alpha}$ diffeomorphism in a way different from the conventional method for $r\geq 2$.
One of the ingredients they provided in \cite{C13} is 
the existence of Lipschitz holonomy along the local unstable and stable laminations, see also \cite[Appendix A]{BCS22}.
Using this,  they also provided the following lemma, where the definition of thickness 
is rather technical and will be given in the next section.

\begin{lem}[Continuity of thickness {\cite[\S4 Proposition~2]{C13}}]\label{cv-lem}
Let $\alpha$ be a real number with $0<\alpha<1$ and 
$f$ a $C^{1+\alpha}$ diffeomorphism having a horseshoe $\Lambda$ on a closed surface $M$. 
The stable thickness 
$\tau^{\rm s}(\Lambda_{g})$ of the continuation $\Lambda_{g}$ of $\Lambda$ depends continuously on $g$ 
in a  $C^{1}$ neighborhood of $f$ on the space of $C^{1+\alpha}$ diffeomorphisms 
such that the  $\alpha$-H\"older norm of 
$Dg$, $Dg^{-1}$ is bounded by $C>0$.
\end{lem}
\noindent
Note that, since the unstable thickness $\tau^{\rm u}(\Lambda_{g})$ is equal to 
$\tau^{\rm s}(\Lambda_{g^{-1}})$, the same  result holds for the unstable thickness. 
In $C^{r}$ topology with $r\geq 2$, the continuity of thickness is shown in \cite{N79, PT93}.

\begin{remark}[Continuity of denseness]\label{rmk2.6}
A slight modification of the definition of thickness yields the concept of denseness, see Definition \ref{defthick}. 
It is therefore easy to see from the proof of \cite[\S4 Proposition~2]{C13} that 
the claim of Lemma \ref{cv-lem} with thickness replaced by denseness is still true.
The denseness will be used in Section \ref{064}.
\end{remark}

\subsection{Open problem and outline}
For future developments, we compare 
\cite[Theorem A]{KS17} with our Theorem~\ref{mainthm}.
The former theorem states that the set $\mathcal{H}$ is dense
in any Newhouse domain  in $\Diff^r(M)$ 
but it does not discuss the property of strong pluripotency. 
On the  other hand, our
Theorem~\ref{mainthm} guarantees that a Newhouse domain in 
$\Diff^r(M)$
contains the open proper subset $\mathcal{U}_0$ of strongly pluripotent 
diffeomorphisms. 
In particular, it follows that $\mathcal{H}$ 
and $\mathcal{D}_{q}$  are both dense in $\mathcal{U}_{0}$.
Thus, the following problem remains open.

\begin{problem}
Is every diffeomorphism in every Newhouse domain strongly pluripotent?
\end{problem}
As noted at the beginning of this section, 
we have obtained all results in this paper by assuming that 
the regularity $r$ of diffeomorphisms is greater than 1. 
Thus, under the $C^1$ regularity constraint, we propose the following problem.

\begin{problem}
Is there a two-dimensional diffeomorphism which is $C^{1}$-robustly (strongly) pluripotent?
\end{problem}

This paper is organized as follows: in Section~\ref{060}, 
we introduce several definitions including 
a model of a wild Smale horseshoe and its $C^{r}$ neighborhood. 
In Section~\ref{092}, we provide some necessary notions and 
properties about the structures of Cantor sets. Sections~\ref{062} 
to \ref{065} are devoted to developing four lemmas which will be used in the proof of Theorem~\ref{mainthm}. Finally, 
Theorem~\ref{HD} is proved in 
Section~\ref{sec8}.

\section{Robust wild horseshoes}\label{060}
In this section, we set the stage for proving the main theorem by introducing a locally linear horseshoe map (which will be denoted by $F$ in our notation) originally introduced by Colli and Vargas. Building upon \cite{CV01}, we provide a preliminary characterization of dynamical features generated by elements in a $C^r$ neighborhood of $F$. Informally speaking, these systems can all be viewed as deformations of $F$ under perturbations.
\subsection{The Colli--Vargas model $F$}\label{112}
To obtain $F_{0}$ in Theorem \ref{mainthm}, we consider 
the so-called Colli--Vargas model. 
Let $F$ be a $C^{r}$  diffeomorphism having a wild horseshoe 
on a closed surface $M$ 
and identical to the one given in \cite{CV01}. 
More precisely, it is defined as follows.
We may suppose that $M$ has a local chart defined on an open set which is identified 
with an open set containing $(-2, 2)^{2}$ of $\mathbb{R}^{2}$.
In this open set, we consider the rectangle $Q=[-1,1]^{2}$ and 
the disjoint vertical strips of $Q$ defined as 
\[
S_{0}=
\left[-\frac{1}{2}-\sigma^{-1}, -\frac{1}{2}+\sigma^{-1}\right]\times [-1,1],\ \ 
S_{1}=
\left[\ \frac{1}{2}-\sigma^{-1}, \frac{1}{2}+\sigma^{-1}\right]\times [-1,1]
\]
for some constant $\sigma>2$. 
We assume that $F\vert_{ S_{0}\cup S_{1}}$ 
satisfies  
\[
F(x,y)=
\left\{\begin{array}{ll}\displaystyle
\left(\sigma\left(x+\frac{1}{2}\right), -\frac{1}{2} +\lambda y \right) & \text{if}\ (x,y)\in S_{0}\\[8pt]
\displaystyle
\left(-\sigma\left(x-\frac{1}{2}\right), \frac{1}{2} -\lambda y\right) & \text{if}\ (x,y)\in S_{1}
\end{array}\right.
\]
for some $\lambda>0$ with 
\begin{equation}\label{lambdasigma}
\lambda\sigma<1. 
\end{equation}
It follows immediately from $\sigma>2$ that we actually have $\lambda<\sigma^{-1}<1/2$. Thus, there is an affine horseshoe for $F$ as 
\begin{equation}\label{hs}
\Lambda=\Lambda_{F}=\bigcap_{n\in \mathbb{Z}}F^{n}(S_{0}\cup S_{1}).
\end{equation}
Then 
$F\vert_{\Lambda}$ is topologically conjugate to  
the full two-sided shift on two symbols  
by the homeomorphism $h=h_{F}: \Lambda_{F}  \to  \{ 0,1\} ^{\mathbb Z}$ 
given by
\[
\left(h(x) \right)_j = w \quad \text{if}\quad F^j(x) \in S_{w}, 
\]
where $(h(x))_j$ is the $j$th entry of $h(x)$. 
Let the fixed point 
$h_{F}^{-1}(\underline{0})$ of $F$ with $\underline{0}=(\ldots0 0 0 \ldots)$ 
be denoted by $p=p_{F}$, 
and hence it satisfies $p=(-a_{\rm u}, -a_{\rm s})$, where $a_{\rm u}=(2(1-\sigma^{-1}))^{-1}$ and $a_{\rm s}=(2(1-\lambda))^{-1}$. 

\vspace{3mm}
Next we consider any $f$ which is $C^{r}$-close to $F$. 
The intersection 
$Q\cap f^{-1}(Q)$ consists of two disjoint components, 
denoted by $S_{0,f}$ and $S_{1,f}$, such that
\[\lim_{f\to F}S_{i,f}=S_{i}\] for each $i\in \{0,1\}$. See Figure  \ref{fig0}.
Then we have the hyperbolic continuation 
$\Lambda_{f}=\bigcap_{n\in \mathbb{Z}}f^{n}(S_{0,f}\cup S_{1,f})$ of $\Lambda_{F}$.
Let us denote by $\lumfd(\Lambda_{f})$ the union of local unstable manifolds $\lumfd(x)$ in $Q$ with $x\in \Lambda_{f}$. 
We write
$B_{0,f}=f(S_{0,f})$, 
$B_{1,f}=f(S_{1,f})$
and denote by $G_{0,f}$
  the component of $Q\setminus \lumfd(\Lambda_{f})$ between $B_{0,f}$ and $B_{1,f}$.
Then, for any component $G_{f}$ of $Q\setminus \lumfd(\Lambda_{f})$
contained in $B_{0,f}\cup B_{1,f}$,
 there exists an integer $n\geq 1$ such that $f^{-n}(G_{f}) \subset G_{0,f}$.
For  such a  $G_{f}$ 
we have two rectangles  $B^{+}_{f}$ and $B^{-}_{f}$ 
which are the connected components of $f^{n}(B_{0,f})\cap Q$ and $f^{n}(B_{1,f})\cap Q$ adjacent to $G_{f}$. 
\begin{figure}[hbt]
\centering
\scalebox{0.84}{
\includegraphics[clip]{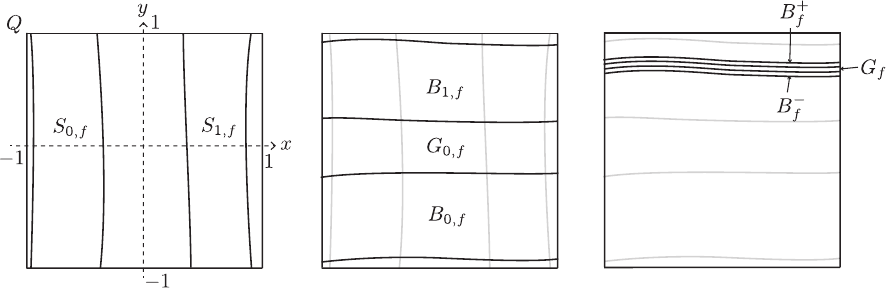}
}
\caption{Strips of $f$ inside $Q$.} 
\label{fig0}
\end{figure}
 
\begin{dfn}\label{defthick}
Let $x$  be a point of $\Lambda_{f}$ 
and 
$\ell$
a connected component of $\lsmfd(x)\setminus  \Lambda_{f}$ contained in $G_{f}$. 
\begin{itemize}
\item 
The \emph{stable thickness} of $\Lambda_{f}$ at $\ell$ is
\begin{equation}\label{ut'}
\tau(\Lambda_{f}, \ell)=\frac{\min\big\{|B_{f}^{-}\cap \lsmfd(x)|,\  |B_{f}^{+}\cap \lsmfd(x)|\big\}}{| {\ell} |},
\end{equation}
where $|\cdot|$ stands for the arc-length of the corresponding arc. Moreover
the \emph{stable thickness} of $\Lambda_{f}$ is defined by
\begin{equation}\label{ut}
\tau^{\rm s}(\Lambda_{f})=\inf_{\ell}\tau(\Lambda_{f}, \ell),
\end{equation}
where the infimum is taken over all connected components $\ell$ of $\lsmfd(x)\setminus  \Lambda_{f}$ 
contained in $B_{0,f}\cup G_{0,f}\cup B_{1,f}$.
\item 
The \emph{unstable thickness} $\tau^{\rm u}(\Lambda_{f})$ of $\Lambda_{f}$ is defined as the stable thickness with respect to $f^{-1}$.
\item 
The constant obtained by replacing `min' with `max' in \eqref{ut'} is called the \emph{stable denseness} of $\Lambda_{f}$ at $\ell$ 
and denoted by $\theta(\Lambda_{f}, \ell)$. 
Moreover, the constant obtained by replacing `inf' with `sup'  
and `$\tau$' with `$\theta$'
in \eqref{ut} 
 is called the \emph{stable denseness} of 
 $\Lambda_{f}$  and denoted by $\theta^{\rm s}(\Lambda_{f})$.
 \item 
The \emph{unstable denseness} $\theta^{\rm u}(\Lambda_{f})$ of $\Lambda_{f}$ is defined as the stable denseness with respect to $f^{-1}$.

\end{itemize}
\end{dfn}

The notion of thickness is often used to show the non-empty intersection of two Cantor sets. On the other hand, if a Cantor set has denseness bounded from above, then, every gap of it occupies a relatively large proportion compared to its adjacent bridges. See Subsection~\ref{071} for the definitions of bridges and gaps. This observation will be helpful in Subsection~\ref{183}.

Since the horseshoe $\Lambda$ given in \eqref{hs}
 is affine, we have 
\begin{equation}\label{thick0}
\tau^{\rm s}(\Lambda)=\theta^{\rm s}(\Lambda)=
\dfrac{\lambda}{1-2\lambda}=:\tau^{\rm s}, \quad 
\tau^{\rm u}(\Lambda)=\theta^{\rm u}(\Lambda)=
\dfrac{\sigma^{-1}}{1-2\sigma^{-1}}=:\tau^{\rm u}.
\end{equation} 
We consider the case 
 that  $\lambda$ and $\sigma$ satisfy the open condition 
\begin{equation}\label{thick1}
\tau^{\rm s}\tau^{\rm u}>1,
\end{equation} which
ensures that $C^{r}$-robust homoclinic tangencies occur by \eqref{f2} below.
Moreover, it follows from \eqref{lambdasigma} that $\lambda<\sigma^{-1}$, thus we have \[\tau^{\rm s}=\dfrac{\lambda}{1-2\lambda}<\dfrac{\sigma^{-1}}{1-2\sigma^{-1}}=\tau^{\rm u},\] which implies, according to \eqref{thick1}, that 
\begin{equation}\label{thick2}
\dfrac{\sigma^{-1}}{1-2\sigma^{-1}}=\tau^{\rm u}>1.
\end{equation}
As a result, the constant $\sigma$ actually satisfies 
\begin{equation}\label{sigma}
2<\sigma<3.
\end{equation}
In such a situation,
we assume that, for any $(x,y)$  in a small neighborhood $U$ of $(0,-a_{\rm s})$ in $(-2,2)^2$,  
\begin{equation}\label{f2}
F^{2}(x,y)=\left(-a_{\rm u}+\mu-\beta x^{2}+\gamma(y+a_{\rm s}), -\alpha x\right),
\end{equation}
where
$\mu$ is positive and will be adjusted later as we need, and 
$\alpha, \beta, \gamma$ are positive constants. See Figure~\ref{fig11}. Then \eqref{f2} preserves the orientation.
This completes $F$'s setup.
\begin{figure}[hbt]
\centering
\scalebox{1.0}{
\includegraphics[clip]{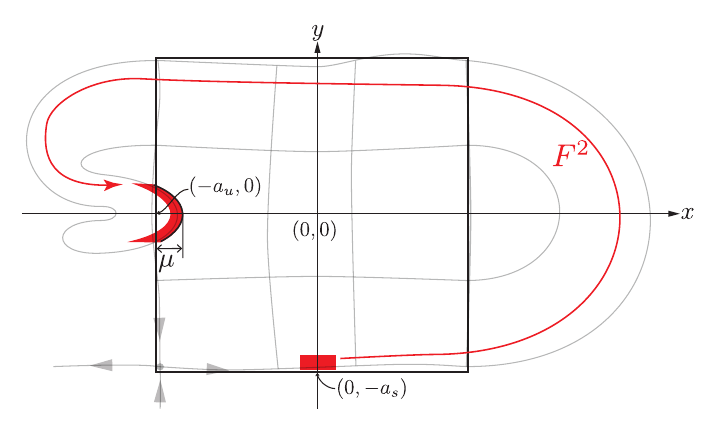}
}
\caption{The rectangle is mapped parabolically by $F^2$.} 
\label{fig11}
\end{figure}
\smallskip

\subsection{A small $C^r$ neighborhood $\mathcal{U}^r_F$ of $F$}\label{110}
Now, we are ready to consider a $C^r$ neighborhood of $F$ in $\Diff^r(M)$.
For a small $\varepsilon_0>0$, we write
\begin{equation}\label{037}
\underline{\lambda}=\lambda-\varepsilon_0,\ \overline{\lambda}=\lambda+\varepsilon_0,\
\underline{\sigma}=\sigma-\varepsilon_0,\ \overline{\sigma}=\sigma+\varepsilon_0.
\end{equation}
According to \eqref{lambdasigma} and \eqref{sigma}, we may suppose that 
\begin{equation}\label{045}
2<\overline{\sigma}<3\quad \text{and} \quad \overline{\lambda}\overline{\sigma}<1
\end{equation}
 by shrinking $\varepsilon_0$ if necessary.

Let $\pi_{x}$ and $\pi_{y}$ be 
the orthogonal projection to the $x$ and $y$-axes, respectively. 
We now consider 
a $C^{r}$-neighborhood $\mathcal{U}^{r}_{F}$ of $F$ in $\Diff^{r}(M)$ depending on $\varepsilon_0$ and 
satisfying  the following conditions:
\begin{align*}
&\sup\left\{
\| D(\pi_{y}\circ f)(x)\|\ :\ f\in \mathcal{U}^{r}_{F},\ x\in S_{0,f}\cup S_{1,f}
\right\}<\overline{\lambda}-\frac{\varepsilon_0}{2},\\
&\inf\left\{
m( D(\pi_{y}\circ f)(x))\ :\ f\in \mathcal{U}^{r}_{F},\ x\in S_{0,f}\cup S_{1,f}
\right\}>\underline{\lambda}+\frac{\varepsilon_0}{2},\\
&\sup\left\{
\| D(\pi_{x}\circ f)(x)\|\ :\ f\in \mathcal{U}^{r}_{F},\ x\in S_{0,f}\cup S_{1,f}
\right\}<\overline{\sigma}-\frac{\varepsilon_0}{2},\\
&\inf\left\{
 m(D(\pi_{x}\circ f)(x))\ :\ f\in \mathcal{U}^{r}_{F},\ x\in S_{0,f}\cup S_{1,f}
\right\}>\underline{\sigma}+\frac{\varepsilon_0}{2},
\end{align*}
where $\|\cdot\|$ and $m(\cdot)$ stand for
the operator and  minimum norms, respectively, of a given linear map. 
Then we may suppose that 
each $f\in \mathcal{U}^{r}_{F}$ has the horseshoe $\Lambda_{f}$ which is the continuation of $\Lambda$. 
Since \eqref{thick1} holds for $F$, by shrinking $\mathcal{U}^r_F$ if necessary, we can suppose that 
\begin{equation}\label{107}
\tau^{\rm s}(\Lambda_f)\tau^{\rm u}(\Lambda_f)>1
\end{equation}
holds for every $f$ in $\mathcal{U}^r_F$. This is because the stable thickness $\tau^{\rm s}(\Lambda_f)$ and the unstable thickness $\tau^{\rm u}(\Lambda_f)$ vary continuously on 
$f$. 
Similarly, note that the stable and unstable denseness also vary continuously on 
$f$, see \cite{N79, PT93}.
Combining this fact  with \eqref{thick0}, we  know that the 
 stable and unstable denseness of $\Lambda$ are positive.
So there exists $\theta=\theta(\mathcal{U}^r_F)>0$ such that
\begin{equation}\label{117}
\max \left\{ \sup_{f\in \mathcal{U}^r_F}\theta^{\rm s}(\Lambda_f), \sup_{f\in \mathcal{U}^r_F}\theta^{\rm u}(\Lambda_f)\right\}<\theta.
\end{equation}
Moreover, 
$f^{2}\vert_{U}$ is given by an expression close to  \eqref{f2} as follows:
\begin{equation}\label{f2'}
f^{2}(x,y)=\left(-\bar a_{\rm u}+\bar\mu-\bar\beta x^{2}+\bar\gamma(y+\bar a_{\rm s}), -\bar\alpha x\right)+h(x,y),
\end{equation}
where each coefficient is close to that in \eqref{f2} and 
$h(x,y)$ stands for the higher order terms containing $o(x^2)$ and $o(y)$.

Let $\mathcal{F}^{\rm s}$ and $\mathcal{F}^{\rm u}$ be 
local stable and unstable foliations  for $\Lambda_{f}$  defined on 
$S_{0,f}\cup S_{1,f}$ and $f(S_{0,f}\cup S_{1,f})$, respectively. These foliations certainly depend on $f$. Hence we also write $\mathcal{F}_{f}^{\rm s}$ and $\mathcal{F}_{f}^{\rm u}$ if the dependence need to be emphasized. 
Then, shrinking $\mathcal{U}^{r}_{F}$ again if necessary,  
by \eqref{thick1} and \eqref{f2'} we may assume that
the intersection between 
leaves of $\mathcal{F}^{\rm s}$ 
and  those of $f^{2}(\mathcal{F}^{\rm u})$
  contains the 
 $C^{1}$ arc of homoclinic quadratic tangencies of $\Lambda_{f}$, 
 denoted by 
 $L$. See \cite{PT93}. We call it a \emph{tangency curve} for simplicity.
On the other hand, the $f^{-2}$-image of  $L$ is denoted by 
 $\widetilde{L}$, see Figure \ref{fig1}. We point out that both $L$ and $\widetilde{L}$ depend on $f$ as well.
\begin{figure}[hbt]
\centering
\scalebox{1.1}{
\includegraphics[clip]{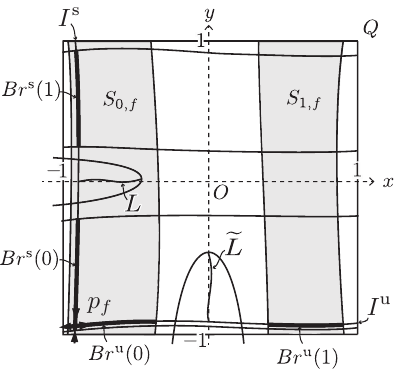}
}
\caption{The tangency curves $L$ and $\widetilde{L}$.} 
\label{fig1}
\end{figure}

 When the parameter $\bar \mu>0$ is fixed first and next it is slid by $\delta$ from  $\bar \mu$, that is, we define the $\delta$-slid perturbation $f_{\delta}$ of $f$ on $f(U)$ by letting
\begin{align*}
f^2_{\delta}(x,y)\
:&=f^2(x,y)+(\delta,0)\\
&=\left(-\bar a_{\rm u}+\bar\mu+\delta-\bar\beta x^{2}+\bar\gamma(y+\bar a_{\rm s}), -\bar\alpha x\right)+h(x,y).
\end{align*}
Since $f^2_{\delta}(\mathcal{F}^{\rm u})$ is slid by $\delta$ along the horizontal direction, we have the new $C^1$ arc of homoclinic tangencies between $f_{\delta}^2(\mathcal{F}^{\rm u})$ and $\mathcal{F}^{\rm s}$, which is denoted by $L(\delta)$. Moreover, we denote $f_{\delta}^{-2}(L(\delta))$ by $\widetilde{L}(\delta)$. It follows immediately that
\[\lim\limits_{\delta\to 0}L(\delta)=L \quad \text{and}\quad \lim\limits_{\delta\to 0}\widetilde{L}(\delta)=\widetilde{L}.\]

\section{Bridges, gaps and bounded distortion}\label{092}
Having established the neighborhood $\mathcal{U}^r_F$ of $F$ in the previous section, in this section, we conduct a detailed analysis of the structure of Cantor sets generated by elements in $\mathcal{U}^r_F$ along their tangency curves.
\subsection{Bridges and gaps}\label{071}
Let $f$ be any diffeomorphism in $\mathcal{U}^{r}_{F}$ 
with the wild Smale horseshoe $\Lambda_{f}$. 
The continuation of the saddle fixed point $p_{F}$ is denoted by $p_{f}$ and 
the connected components
$\smfd(p_{f})\cap Q$ and $\umfd(p_{f})\cap Q$ containing $p_{f}$ 
is denoted by $I^{\rm s}$ and $I^{\rm u}$, respectively. See Figure \ref{fig1}. 
Sometimes we also write $I^{\rm s}_f$ and $I^{\rm u}_f$ if their dependence on $f$ needs to be emphasized. Then we have two Cantor sets 
\begin{equation}\label{111}
\Lambda^{\rm s}_{f} =\Lambda_{f}\cap  I^{\rm s}
\quad \text{and}\quad  
\Lambda^{\rm u}_{f} =\Lambda_{f}\cap  I^{\rm u}.
\end{equation}
For these two Cantor sets, in a similar way as that in Definition \ref{defthick}, one can also define their thicknesses $\tau(\Lambda^{\rm s}_f)$ and  $\tau(\Lambda^{\rm u}_f)$. One can deduce that 
\[\tau(\Lambda^{\rm s}_f)=\tau^{\rm s}(\Lambda_f) \quad\text{and}\quad \tau(\Lambda^{\rm u}_f)=\tau^{\rm u}(\Lambda_f).\] 
The notion of thickness plays an important role when we are aiming to find the intersection of two given Cantor sets. Precisely, the following so-called Gap Lemma is quite helpful.

\begin{lem}[Gap Lemma \cite{N79,PT93}]\label{gaplemma}
Let $K_1, K_2$ be Cantor sets with thicknesses $\tau_1$ and $\tau_2$. If $\tau_1 \tau_2>1$, then one of the following three alternatives occurs: $K_1$ is contained in a gap of $K_2$; $K_2$ is contained in a gap of $K_1$; $K_1\cap K_2\not= \emptyset$.
\end{lem}

We now introduce 
bridges and gaps related to $\Lambda^{\rm s}_{f}$ and 
 $\Lambda^{\rm u}_{f}$. 
For each $i\in \{-1,0,1\}$, 
let $I^{\rm s}_{i}$ and 
$I^{\rm u}_{i}$ be the component of 
$I^{\rm s}\setminus \Lambda_{f}^{\rm s}$ 
and
$I^{\rm u}\setminus \Lambda_{f}^{\rm u}$
 such that 
$I^{\rm s}_{i}\cap \{y=i\}\neq \emptyset$ 
and 
$I^{\rm u}_{i}\cap \{x=i\}\neq \emptyset$.
Let $(Br^{\rm s}(0), Br^{\rm s}(1))$  be 
the pair of components of $I^{\rm s }\setminus (I^{\rm s }_{-1} \cup I^{\rm s}_{0} \cup I^{\rm s }_{1})$ 
and 
$(Br^{\rm  u}(0), Br^{\rm u}(1))$
the pair of components of $I^{\rm u}\setminus (I^{\rm u}_{-1} \cup I^{\rm u}_{0} \cup I^{\rm u}_{1})$
such that 
\[\pi_{y}(Br^{\rm s}(0))<0<\pi_{y}(Br^{\rm s}(1)), \quad
\pi_{x}(Br^{\rm u}(0))<0<\pi_{x}(Br^{\rm u}(1)). 
\]
Next, we consider other projections 
\begin{equation}\label{116}
\pi_{\mathcal{F}^{\rm s}}:S_{0,f}\cup S_{1,f}\to I^{\rm u},\quad
\pi_{\mathcal{F}^{\rm u}}: f(S_{0,f}\cup S_{1,f})\to I^{\rm s},
\end{equation}
where the former is along the leaves of $\mathcal{F}^{\rm s}$ and 
the latter is along the leaves of $\mathcal{F}^{\rm u}$. 
 Since both
 $\mathcal{F}^{\rm s}$ and $\mathcal{F}^{\rm u}$ are $C^{1}$-foliations and every leaf of them transversely meets $I^{\rm u}$ and $I^{\rm s}$, respectively,  
 $\pi_{\mathcal{F}^{\rm s}}$ and $\pi_{\mathcal{F}^{\rm u}}$ are 
$C^{1}$-submersions. Then
we have two pairs of horizontal strips and vertical strips in $Q$ defined as 
\begin{align*}
& \mathbb{B}r^{\rm s}(0) = (\pi_{\mathcal{F}^{\rm u}})^{-1}(Br^{\rm s}(0) ),\  \ 
\mathbb{B}r^{\rm s}(1) = (\pi_{\mathcal{F}^{\rm u}})^{-1}(Br^{\rm s}(1)),\\
& \mathbb{B}r^{\rm u}(0) = (\pi_{\mathcal{F}^{\rm s}})^{-1}(Br^{\rm u}(0) ),\  \ 
\mathbb{B}r^{\rm u}(1) = (\pi_{\mathcal{F}^{\rm s}})^{-1}(Br^{\rm u}(1)).
\end{align*}
Note that 
$\mathbb{B}r^{\rm s}(w)\subset f(S_{w,f})$  and 
$\mathbb{B}r^{\rm u}(w)\subset S_{w,f}$ 
 for each $w \in\{0,1\}$. See Figure \ref{fig1}.

For every integer $n \geq 1$, let $\underline{w}$ be a binary code of $n$ entries, that is, 
$\underline{w}= (w_{1} \dots w_{n}) \in \{0,1\}^{n}$. For such $n$ and $w_{n}$, we define 
\begin{align*}
\mathbb{B}r^{\rm s}(n;\underline{w})&=\left\{ x \in  Q\ :\  f^{-i+1}(x) \in  \mathbb{B}r^{\rm s}(w_{i}),\ i = 1,\dots,n\right\},\\
\mathbb{B}r^{\rm u}(n;\underline{w})&=\left\{ x \in  Q\ :\  f^{i-1}(x) \in  \mathbb{B}r^{\rm u}(w_{i}),\ i = 1,\dots,n\right\},
\end{align*}
and hence $\mathbb{B}r^{\rm s}(1; w)=\mathbb{B}r^{\rm s}(w)$ 
and 
$\mathbb{B}r^{\rm u}(1; w)=\mathbb{B}r^{\rm u}(w)$
 for each $w\in \{0,1\}$. 
 Given $n\in\mathbb{N}$ and $\underline{w}\in \{0,1\}^{n}$, 
 we call $\mathbb{B}r^{\rm s}(n;\underline{w})$ the 
\emph{s-bridge strip} and $\mathbb{B}r^{\rm u}(n;\underline{w})$ the 
\emph{u-bridge strip}. See Figure \ref{fig2}.
\begin{figure}[hbt]
\centering
\scalebox{1.1}{
\includegraphics[clip]{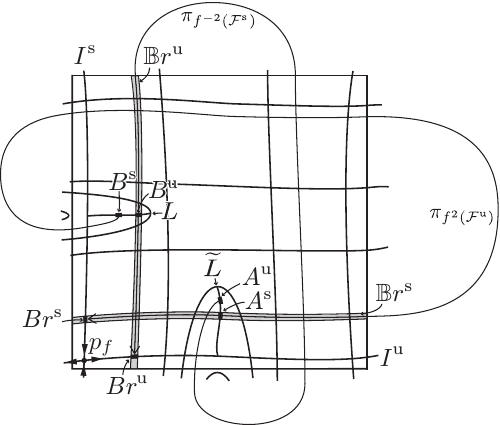}
}
\caption{The notation $(n;\underline{w})$ of each bridge and strip is omitted.} 
\label{fig2}
\end{figure}
Observe that, for each integer $n\geq 1$,  
$(\mathbb{B}r^{\rm s}(n;\underline{w}))_{\underline{w} \in  \{0,1\}^{n}}$ and 
$(\mathbb{B}r^{\rm u}(n;\underline{w}))_{\underline{w} \in  \{0,1\}^{n}}$
consist of $2^{n}$ mutually disjoint horizontal and vertical strips, respectively. It is easy to verify that 
\[f^{n}(\mathbb{B}r^{\rm u}(n;w_1\dots w_n))=\mathbb{B}r^{\rm s}(n;w_n\dots w_1)=\mathbb{B}r^{\rm s}(n;[w_1\dots w_n]^{-1}).\]
Moreover, we set 
\begin{equation}\label{135}
\begin{split}
Br^{\rm s}(n;\underline{w})&=\mathbb{B}r^{\rm s}(n;\underline{w}) \cap  I^{\rm s} = \pi_{\mathcal{F}^{\rm u}} (\mathbb{B}r^{\rm s}(n;\underline{w})),\\
Br^{\rm u}(n;\underline{w})&=\mathbb{B}r^{\rm u}(n;\underline{w}) \cap  I^{\rm u} = 
\pi_{\mathcal{F}^{\rm s}}(\mathbb{B}r^{\rm u}(n;\underline{w})),
\end{split}
\end{equation} 
which are called  
\emph{{\rm s}-bridge of $\Lambda^{\rm s}_f$} and 
\emph{{\rm u}-bridge of $\Lambda^{\rm u}_f$}, respectively.
In these notations, $n$ is called the \emph{generation} and $\underline{w}$ the \emph{itinerary} for the corresponding bridges and bridge strips.
The \emph{length} of $\underline{w}$, denoted by $|\underline{w}|$, 
is defined as the cardinality of binary codes that comprise $\underline{w}$, 
that is, $|\underline{w}|=|(w_{1} \dots w_{n})|=n$.

Next,   
the maximum subinterval of $Br^{\rm s}(n;\underline{w})$ 
 between $Br^{\rm s}(n+1;\underline{w}0)$ and $Br^{\rm s}(n+1;\underline{w}1)$ 
 is denoted by $Ga^{\rm s}(n;\underline{w})$, 
 while 
 the maximum subinterval of $Br^{\rm u}(n;\underline{w})$ 
 between $Br^{\rm u}(n+1;\underline{w}0)$ and $Br^{\rm u}(n+1;\underline{w}1)$ 
 is denoted by $Ga^{\rm u}(n;\underline{w})$, 
 which are respectively called the \emph{{\rm s}-gap} and \emph{{\rm u}-gap} of generation $n$ and itinerary $\underline{w}$.
 The two bridges $Br^{{\rm s}(\rm u)}(n+1;\underline{w}0)$ and $Br^{{\rm s}(\rm u)}(n+1;\underline{w}1)$ 
 are called {\it adjacent ${\rm s}({\rm u})$-bridges} of $Ga^{{\rm s}(\rm u)}(n;\underline{w})$. If it is necessary to specify the diffeomorphism $f$ concerning the ${\rm s(u)}$-bridge and gap, we may write $Br^{\rm s(u)}_{f}(n;\underline{w})$ and $Ga^{\rm s(u)}_{f}(n;\underline{w})$.

To introduce the ${\rm s(u)}$-bridges and ${\rm s(u)}$-gaps on the tangency curves $L$ and $\widetilde{L}$, we consider extended projections
\[
\pi_{f^{-2}(\mathcal{F}^{\rm s})}:f^{-2}(S_{0,f}\cup S_{1,f})\to I^{\rm u},\quad 
\pi_{f^{2}(\mathcal{F}^{\rm u})}:f^{2}( f(S_{0,f}\cup S_{1,f}) )\to I^{\rm s},
\]
where the former is the projection along the leaves of $f^{-2}(\mathcal{F}^{\rm s})$ and 
the latter is that along the leaves of $f^{2}(\mathcal{F}^{\rm u})$. 
Let 
$B^{\rm s}(n;\underline{w})$  
and 
$B^{\rm u}(n;\underline{w})$ 
be the sub-arcs of $L$ with the following conditions: 
\begin{equation}\label{113}
Br^{\rm s}(n;\underline{w})=\pi_{f^{2}(\mathcal{F}^{\rm u})}(B^{\rm s}(n;\underline{w})),\quad
Br^{\rm u}(n;\underline{w})=\pi_{\mathcal{F}^{\rm s}}(B^{\rm u}(n;\underline{w})),
\end{equation}
which are called \emph{{\rm s}} and \emph{{\rm u}-bridges} on $L$, respectively. Moreover, in the same manner, the two Cantor sets $\Lambda^{\rm s}_f$ and $\Lambda^{\rm u}_f$ defined in \eqref{111} also have their projections $\Lambda^{\rm s}_L$ and $\Lambda^{\rm u}_L$ on $L$ defined by 
\begin{equation}\label{138}
\Lambda^{\rm s}_f=\pi_{f^2(\mathcal{F}^{\rm u})}(\Lambda^{\rm s}_L)\quad \mbox{and} \quad \Lambda^{\rm u}_f=\pi_{\mathcal{F}^{\rm s}}(\Lambda^{\rm u}_L).
\end{equation}
On the other hand, the
\emph{{\rm s}} and \emph{{\rm u}-bridges} $A^{\rm s}(n;\underline{w})$ 
and 
$A^{\rm u}(n;\underline{w})$ on $\widetilde{L}$ are defined by 
\begin{equation}\label{114}
Br^{\rm s}(n;\underline{w})= 
\pi_{\mathcal{F}^{\rm u}}(A^{\rm s}(n;\underline{w})), \quad
Br^{\rm u}(n;\underline{w})= 
\pi_{f^{-2}(\mathcal{F}^{\rm s})}(A^{\rm u}(n;\underline{w})),
\end{equation}
respectively. Similarly, we can also define ${\rm s(u)}$-gaps on $L$ and $\widetilde{L}$ respectively. For instance, the maximum subinterval of $B^{\rm s}(n;\underline{w})$ 
 between $B^{\rm s}(n+1;\underline{w}0)$ and $B^{\rm s}(n+1;\underline{w}1)$ 
 is denoted by $G^{\rm s}(n;\underline{w})$, 
 while 
 the maximum subinterval of $B^{\rm u}(n;\underline{w})$ 
 between $B^{\rm u}(n+1;\underline{w}0)$ and $B^{\rm u}(n+1;\underline{w}1)$ 
 is denoted by $G^{\rm u}(n;\underline{w})$, 
 which are respectively called the \emph{{\rm s}-gap} and \emph{{\rm u}-gap} of generation $n$ and itinerary $\underline{w}$ on $L$.

\subsection{Bounded distortion of bridges}\label{061}
The following lemma and its remark are useful when we estimate the ratios of the lengths of bridges with different generations. Recall that $\mathcal{U}^r_F$ is the small neighborhood of $F$ given in Subsection \ref{110} and $\underline{\lambda}$, $\overline{\lambda}$, $\underline{\sigma}$, $\overline{\sigma}$ are constants defined in \eqref{037}.
For any $f\in \mathcal{U}^r_F$, the \emph{length} of an arc $J$ in $I^{\rm s}_{f}$, $I^{\rm u}_{f}$, $L_{f}$ or $\widetilde{L}_{f}$ means its arc-length, which is denoted by $|J|$.

\begin{lem}\label{lem4}
For any $f\in \mathcal{U}^r_F$, $n\in \mathbb{N}$  and $\underline{w}\in \{0,1\}^{n}$, 
let $Br^{\rm s}(n;\underline{w})$  
be the {\rm s}-bridge of $\Lambda^{\rm s}_f$
and $Br^{\rm u}(n;\underline{w})$ the {\rm u}-bridge of $\Lambda^{\rm u}_f$, respectively. 
Then, for $i=0,1$, we have
\begin{align}
\ubar{\lambda}&\le \dfrac{|Br^{\rm s}(n+1;\underline{w}i)|}{|Br^{\rm s}(n;\underline{w})|} \le \overline{\lambda};\\
\overline{\sigma}^{-1}&\le \dfrac{|Br^{\rm u}(n+1;\underline{w}i)|}{|Br^{\rm u}(n;\underline{w})|} \le \ubar{\sigma}^{-1}.
\end{align}
\end{lem}

\begin{pf}
See \cite[Lemma 4.1]{KS17} for the proof. 
\end{pf}

\begin{remark}\label{139}
When the generation $n$ is sufficiently large, since 
$\pi_{f^{2}(\mathcal{F}^{\rm u})}$ and $\pi_{\mathcal{F}^{\rm s}}$ are almost affine, the same conclusion also holds for s-bridges and u-bridges on $L$. Precisely, for $i=0,1$, we have that 
\begin{align}
\ubar{\lambda}&\le \dfrac{|B^{\rm s}(n+1;\underline{w}i)|}{|B^{\rm s}(n;\underline{w})|} \le \overline{\lambda};\label{073}\\
\overline{\sigma}^{-1}&\le \dfrac{|B^{\rm u}(n+1;\underline{w}i)|}{|B^{\rm u}(n;\underline{w})|} \le \ubar{\sigma}^{-1}.
\end{align}
\end{remark}

Let $B^{\rm s}$ be an ${\rm s}$-bridge on $L$. 
For a given $\delta$ with $|\delta|$ sufficiently small, 
the bridge in the slid tangency curve $L(\delta)$ with the same itinerary as that of $B^{\rm s}$ is denoted by $B^{\rm s}(\delta)$ and called the \emph{$\delta$-slid ${\rm s}$-bridge} of $B^{\rm s}$.
The \emph{$\delta$-slid ${\rm u}$-bridge} is defined in a similar way. 
Sometimes we also call them  $\delta$-slid bridges for simplicity. Obviously, $\delta$-slid bridges are bridges associated to $f_{\delta}$. It follows immediately from the definition that 
\begin{equation}\label{115}
\lim\limits_{\delta\to 0}B^{\rm s}(\delta)=B^{\rm s} \quad \text{and }\quad \lim\limits_{\delta\to 0}B^{\rm u}(\delta)=B^{\rm u}.
\end{equation}

The next lemma indicates that $|B^{\rm s}|$ and $|B^{\rm s}(\delta)|$ (also $|B^{\rm u}|$ and $|B^{\rm u}(\delta)|$) do not differ a lot when $|\delta|$ is small.

\begin{lem}\label{lem3}
For any $f\in \mathcal{U}^r_F$ and $\delta\in\mathbb{R}$ close to $0$, 
there is a constant $c>0$  such that, 
for any 
{\rm s}-bridge $B^{\rm s}$ and {\rm u}-bridge $B^{\rm u}$ on $L$
and their
 $\delta$-slid bridges $B^{\rm s}(\delta)$ and  $B^{\rm u}(\delta)$ on $L(\delta)$, 
their lengths satisfy the following length estimations.
\begin{enumerate}[\rm (1)]
\item $(1-c|\delta|)|B^{\rm s}|\leq |B^{\rm s}(\delta)|\leq (1+c|\delta|)|B^{\rm s}|$,
\item $(1-c|\delta|)|B^{\rm u}|\leq |B^{\rm u}(\delta)|\leq (1+c|\delta|)|B^{\rm u}|$ and
\item
$(1-c|\delta|)|B^{\rm s}\cap B^{\rm u}|-\kappa|\delta|
\leq  |B^{\rm s}(\delta)\cap B^{\rm u}(\delta)| \leq(1+c|\delta|)|B^{\rm s}\cap B^{\rm u}|+\kappa|\delta|,$
\end{enumerate}
where $\kappa>2$ is some constant independent of $f$ and $\delta$.
\end{lem}

\begin{figure}[hbt]
\centering
\scalebox{1.0}{
\includegraphics[clip]{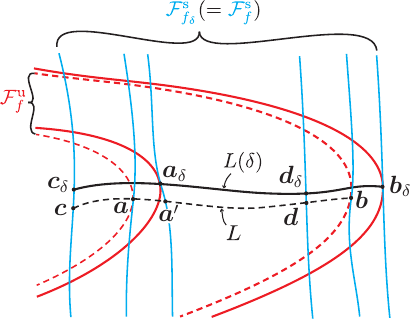}
}
\caption{Bridges on $L$ and $L(\delta)$.} 
\label{fig3}
\end{figure}
\begin{proof}
Let $\boldsymbol{a},\boldsymbol{b}\in L$ be the endpoints of $B^{\rm u}$ and $\boldsymbol{a}_{\delta},\boldsymbol{b}_{\delta}\in L(\delta)$ the endpoints of $B^{\rm u}(\delta)$. The interval in $L$ between $\boldsymbol{a}$ and $\boldsymbol{b}$ and that in $L_\delta$ between $\boldsymbol{a}_\delta$ and $\boldsymbol{b}_\delta$ are denoted by $\boldsymbol{a}\boldsymbol{b}$ and $\boldsymbol{a}_\delta\boldsymbol{b}_\delta$, respectively. Obviously, by \eqref{115}, we have
\[\lim\limits_{\delta\to 0}\boldsymbol{a}_{\delta}=\boldsymbol{a} \quad \text{and} \quad\lim\limits_{\delta\to 0}\boldsymbol{b}_{\delta}=\boldsymbol{b}, \]
because the leaves of $\mathcal{F}^{\rm s}_f$ and $\mathcal{F}^{\rm u}_f$ depend continuously on $f$. See Figure \ref{fig3}.
Moreover, since $L(\delta)$ $C^1$ converges to $L$ as $\delta\to 0$, there exists a constant $c>0$ independent of $\delta$ such that
\[(1-c|\delta|)|\boldsymbol{a}\boldsymbol{b}|_L\le |\boldsymbol{a}_{\delta}\boldsymbol{b}_{\delta}|_{L(\delta)}\le(1+c|\delta|)|\boldsymbol{a}\boldsymbol{b}|_L,\]
where we denote the arc-lengths of the bridge $B^{\rm u}\subset L$ and $B^{\rm u}(\delta)\subset L(\delta)$ by $|\boldsymbol{a}\boldsymbol{b}|_L$ and $|\boldsymbol{a}_{\delta}\boldsymbol{b}_{\delta}|_{L(\delta)}$ respectively. This proves (2), while (1) can be shown similarly. 

Now, let us prove (3). Under the same notations as above, let $\boldsymbol{c}, \boldsymbol{d}\in L$ be the endpoints of $B^{\rm s}$ and $\boldsymbol{c}_{\delta},\boldsymbol{d}_{\delta}\in L(\delta)$ the endpoints of $B^{\rm s}(\delta)$. Thus, we have
\[|B^{\rm s}\cap B^{\rm u}|=|\boldsymbol{a}\boldsymbol{d}|_L  \quad\text{and}\quad   |B^{\rm s}(\delta)\cap B^{\rm u}(\delta) |=|\boldsymbol{a}_\delta \boldsymbol{d}_\delta|_{L(\delta)}. \]
Let $\boldsymbol{a}'\in L$ be the intersection of $\mathcal{F}^{\rm s}_{f_{\delta}}(\boldsymbol{a}_{\delta})$ and $L$, where $\mathcal{F}^{\rm s}_{f_{\delta}}(\boldsymbol{a}_{\delta})$ is the leaf of $\mathcal{F}^{\rm s}_{f_{\delta}}$ 
passing 
through $\boldsymbol{a}_\delta$. Since $\boldsymbol{a}_\delta \boldsymbol{d}_\delta\subset L(\delta)$ is the $\delta$-slid segment of $\boldsymbol{a}'\boldsymbol{d}\subset L$, by applying the same argument as that in the proof of (1) to $\boldsymbol{a}_\delta \boldsymbol{d}_\delta$ and $\boldsymbol{a}'\boldsymbol{d}$, we get
\begin{equation}\label{066}
|\boldsymbol{a}_\delta \boldsymbol{d}_\delta|_{L(\delta)} \ge (1-c|\delta|)|\boldsymbol{a}'\boldsymbol{d}|_L=(1-c|\delta|)\big(|\boldsymbol{a}\boldsymbol{d}|_L-|\boldsymbol{a}\boldsymbol{a}'|_L\big).
\end{equation}
On the other hand, again by the $C^1$ dependence of $\boldsymbol{a}_\delta$ on $\delta$, we have
\begin{equation}\label{067}
|\boldsymbol{a}\boldsymbol{a}'|_L\le \kappa|\delta|
\end{equation}
for some constant $\kappa>2$ which does not depend on $f$ and $\delta$. Now, \eqref{066} and \eqref{067} together imply
\[|B^{\rm s}(\delta)\cap B^{\rm u}(\delta)|\geq(1-c|\delta|)|B^{\rm s}\cap B^{\rm u}|-\kappa|\delta|+c\kappa \delta^2.\]
By replacing $\kappa$ with a larger number (still denoted by $\kappa$) if necessary, we obtain the first inequality of (3) because $c\kappa \delta^2$ is much smaller than $\kappa|\delta|$ when $\delta$ is small enough. The other inequality of (3) can be proven similarly. This completes the proof of Lemma \ref{lem3}.
\end{proof}

\section{Linking Lemma}\label{062}
The key result of this section, Lemma \ref{lem1} (Linking Lemma), will be repeatedly invoked in the next section. It tells us that, if we have a linked pair of bridges, then by the $\delta$-slid perturbation with $|\delta|$ arbitrarily small, we can obtain two new linked pairs that correspond to the sub-bridges of the original ones. 

Fix two bridges $B_1$ and $B_2$ on $L$, we say that they are {\it linked} or $(B_1, B_2)$ is a {\it linked pair} if $B_1\cap B_2\not=\emptyset$ and neither $B_1$ is contained in the interior of any gap of $B_2$ nor $B_2$ is contained in the interior of any gap of $B_1$. Suppose $(B_1, B_2)$ is a linked pair. 
\begin{itemize}
\item For a given $\xi>0$, we say that $B_1$ and $B_2$ are {\it $\xi$-linked} if \[|B_1\cap B_2|\ge\xi\min \{|B_1|,|B_2|\}.\]
\item We say that $(B_1,B_2)$ is {\it proportional} if there exists a constant $K\in(0,1)$ independent of $B_1$ and $B_2$ such that either
\[K|B_1|\le |B_2| \le |B_1|\quad \text{or} \quad K|B_2|\le |B_1| \le |B_2|\]
holds.
\end{itemize}   
Two s-bridges $B_1^{\rm s}$ and $B_2^{\rm s}$ are called {\it related} if they are the two maximal proper sub-bridges 
of another ${\rm s}$-bridge $B^{\rm s}$. In this case, the gap of $B^{\rm s}$ which lies between $B_1^{\rm s}$ and $B_2^{\rm s}$ is called the {\it center gap} of $B^{\rm s}$. Similar definitions can be given for ${\rm u}$-bridges and ${\rm u}$-gaps.

Let $\xi_0$ be the constant defined as
\begin{equation}\label{xi}
\xi_0:=\dfrac{(\overline{\sigma}+2)(3-\overline{\sigma})}{3(\overline{\sigma}+3)}\in (0,1),
\end{equation} which only depends on the neighborhood  $\mathcal{U}_F^r$ above. The following lemma plays an important role in the next section. The proof is based on similar versions in \cite{CV01} and \cite{KS17}. 

\begin{lem}[Linking Lemma]\label{lem1}
For every $f\in \mathcal{U}_F^r$, suppose $(B^{\rm s},B^{\rm u})$ is a linked pair. For every $\varepsilon>0$, there exist $\delta$ with $|\delta|<\varepsilon$, related sub-bridges $B^{\rm s}_1, B^{\rm s}_2$ of $B^{\rm s}$ and $B^{\rm u}_1, B^{\rm u}_2$ of $B^{\rm u}$ such that the pair of $\delta$-slid bridges $(B^{\rm s}_i(\delta),B^{\rm u}_i(\delta))$ is $\xi_0$-linked for  $i=1,2$.
\end{lem}

\begin{proof}
Recall that $B^{\rm s (\rm u)}(\delta)$ is the $\delta$-slid bridge of $B^{\rm s(\rm u)}$. Thus, $B^{\rm s (\rm u)}(0)$ is exactly $B^{\rm s(\rm u)}$. In the following proof, we write $B^{\rm s (\rm u)}(0)$ instead of $B^{\rm s(\rm u)}$ for notation consistence.
Let $c$ be the constant in Lemma \ref{lem3}. Fix an arbitrarily small $\varepsilon>0$. In particular, we can assume that 
\begin{equation}\label{001}
\dfrac{1-c\varepsilon}{(1+c\varepsilon)^2}>\ubar{\lambda} \quad \text{and} \quad \left(\dfrac{1+c\varepsilon}{1-c\varepsilon}\right)^2<\dfrac{1}{\overline{\sigma}-2}.
\end{equation}
These inequalities hold for sufficiently small $\varepsilon>0$ because
$\overline{\lambda}$ and $\overline{\sigma}$ are sufficiently close to $\lambda$ and $\sigma$, respectively, which satisfy $0<\lambda<\dfrac{1}{2}$ and $2<\sigma<3$, see \eqref{lambdasigma} and \eqref{sigma}. Denote
\begin{equation}\label{003}
\lambda_0:=\dfrac{\ubar{\lambda}}{1+c\varepsilon} \quad \text{and}\quad \sigma_0:=\overline{\sigma}\left(\dfrac{1+c\varepsilon}{1-c\varepsilon}\right)^2.
\end{equation}
Note that when $\varepsilon>0$ is taken sufficiently small in advance, these constants $\lambda_0$ and $\sigma_0$ can be made as close to $\ubar{\lambda}$ and $\overline{\sigma}$ as we want. Hence, we have
\begin{equation}\label{006}
0<\lambda_0<\ubar{\lambda}<\overline{\lambda}<1<2<\ubar{\sigma}<\overline{\sigma}<\sigma_0<\dfrac{\overline{\sigma}+3}{2}<3.
\end{equation}
 In addition, by the same reason, we can always assume that
\begin{equation}\label{002}
\lambda_0\left(1+\dfrac{1-c\varepsilon}{1+c\varepsilon}\sigma_0\right)<2.
\end{equation}
To see this, since $\lambda\sigma<1$ and $2<\sigma<3$, we have
\begin{equation}\label{024}
\lambda(1+\sigma)<\dfrac{1+\sigma}{\sigma}<\dfrac{1+3}{2}=2.
\end{equation}
By shrinking $\varepsilon_0$ in \eqref{037} if necessary, we can make $\overline{\lambda}$ and $\ubar{\lambda}$ (resp. $\overline{\sigma}$ and $\ubar{\sigma}$) sufficiently close to $\lambda$ (resp. $\sigma$). As a result of \eqref{024}, we have
\begin{equation}\label{046}
\overline{\lambda}(1+\overline{\sigma})<2.
\end{equation}
Therefore, \eqref{002} follows directly from \eqref{046} together with the smallnesses of $|\sigma_0-\overline{\sigma}|$, $|\lambda_0-\ubar{\lambda}|$ and $\varepsilon$.

Note that $\Lambda^{\rm s}_{f_\delta}$ and $\Lambda^{\rm u}_{f_\delta}$ are almost affine images of $\Lambda^{\rm s}_{L(\delta)}$ and $\Lambda^{\rm u}_{L(\delta)}$ under $\pi_{f^2_{\delta}(\mathcal{F}^{\rm u}_{f_\delta})}$ and $\pi_{\mathcal{F}^{\rm s}_{f_\delta}}$ respectively (see \eqref{138}), if $|B^{\rm s}(\delta)|$ and $|B^{\rm u}(\delta)|$ are sufficiently small, according to \eqref{107}, we have
\begin{equation}\label{068}
\tau^{\rm s}_{f_{\delta}}(\Lambda^{\rm s}_{L(\delta)}\cap B^{\rm s}(\delta))\cdot \tau^{\rm u}_{f_{\delta}}(\Lambda^{\rm s}_{L(\delta)}\cap B^{\rm u}(\delta))>1
\end{equation} 
for every $\delta$ with $|\delta|$ small enough. Thus, by applying Lemma \ref{gaplemma} (Gap Lemma) to $\Lambda^{\rm s}_{L(0)}\cap B^{\rm s}(0)$ and $\Lambda^{\rm u}_{L(0)}\cap B^{\rm u}(0)$, there are sub-bridge $\widehat{B}^{\rm s}(0)$ of $B^{\rm s}(0)$ and sub-bridge $\widehat{B}^{\rm u}(0)$ of $B^{\rm u}(0)$ with lengths 
\[|\widehat{B}^{\rm s}(0)|=:\widehat{b}_{\rm s}\quad \text{and} \quad
|\widehat{B}^{\rm u}(0)|=:\widehat{b}_{\rm u}\] satisfying
\begin{itemize}
\item $\widehat{B}^{\rm s}(0)$ and $\widehat{B}^{\rm u}(0)$ have a common point,
\item $\lambda_0^2\dfrac{\varepsilon}{2}<\widehat{b}_{\rm s}<\lambda_0\dfrac{\varepsilon}{2}$ and
\item $\dfrac{1+c\varepsilon}{1-c\varepsilon}\widehat{b}_{\rm s}\le\widehat{b}_{\rm u}<\dfrac{1-c\varepsilon}{1+c\varepsilon}\sigma_0\widehat{b}_{\rm s}$.
\end{itemize}
Indeed, the second item holds because the interval $[\lambda_0^2 \varepsilon,\lambda_0 \varepsilon]$ contains a contracting fundamental domain with the contracting rate $\lambda_0$ which is stronger than $\ubar{\lambda}$. The third item holds because the interval $\left[\dfrac{1+c\varepsilon}{1-c\varepsilon}\widehat{b}_{\rm s},\dfrac{1-c\varepsilon}{1+c\varepsilon}\sigma_0\widehat{b}_{\rm s}\right]$ contains an expanding fundamental domain with expanding rate $\left(\dfrac{1-c\varepsilon}{1+c\varepsilon}\right)^2\sigma_0$ which is equal to $\overline{\sigma}$.

Notice that by \eqref{002}, we have
\begin{equation}\label{069}
\widehat{b}_{\rm s}+\widehat{b}_{\rm u}<\widehat{b}_{\rm s}\left(1+\dfrac{1-c\varepsilon}{1+c\varepsilon}\sigma_0\right)<\lambda_0\dfrac{\varepsilon}{2}\left(1+\dfrac{1-c\varepsilon}{1+c\varepsilon}\sigma_0\right)<\varepsilon.
\end{equation}
Let us consider 
the $\delta$-slid perturbation 
$f_{\delta}$ of $f$ with $|\delta|<\varepsilon$ such that the center of $\widehat{G}^{\rm s}(\delta)$ coincides with the center of $\widehat{G}^{\rm u}(\delta)$, where $\widehat{G}^{{\rm s(u)}}(\delta)$ is the center gap of $\widehat{B}^{{\rm s(u)}}(\delta)$. Let us denote the two related bridges (from left to right) of $\widehat{B}^{\rm s}(\delta)$ (resp. $\widehat{B}^{\rm u}(\delta)$) by $B^{\rm s}_1(\delta)$ and $B^{\rm s}_2(\delta)$ (resp. $B^{\rm u}_1(\delta)$ and $B^{\rm u}_2(\delta)$). We see that $B^{\rm s}_1(\delta)$ and $B^{\rm s}_2(\delta)$ (resp. $B^{\rm u}_1(\delta)$ and $B^{\rm u}_2(\delta)$) have the same generation because they are related bridges.

The following claim gives us useful information on the size comparison among these bridges and gaps whose proof will be postponed until we finish the proof of the lemma.  

\begin{clm}\label{clm1}
With the notations defined above, the following inequalities hold:
\begin{itemize}
\item[(1)] $\ubar{\lambda}^3\varepsilon/2<|\widehat{B}^{\rm s}(\delta)|<\ubar{\lambda}\varepsilon/2$,
\item[(2)] $|\widehat{B}^{\rm s}(\delta)|\le|\widehat{B}^{\rm u}(\delta)|<\sigma_0|\widehat{B}^{\rm s}(\delta)|$,
\item[(3)] $|\widehat{G}^{\rm u}(\delta)|<|\widehat{B}^{\rm s}(\delta)|$.
\end{itemize}
\end{clm}

Let us continue the proof of Lemma \ref{lem1}. We show that for $i=1,2$, these pairs $(B^{\rm s}_i(\delta),B^{\rm u}_i(\delta))$ are proportional. Indeed, we have
\begin{equation}\label{004}
|B^{\rm s}_i(\delta)|\le\overline{\lambda}|\widehat{B}^{\rm s}(\delta)|\le\overline{\lambda}|\widehat{B}^{\rm u}(\delta)|<\overline{\sigma}^{-1}|\widehat{B}^{\rm u}(\delta)|\le|B^{\rm u}_i(\delta)|,
\end{equation}
where the second inequality follows from item (2) of Claim \ref{clm1} and the third inequality follows from the assumption \eqref{045}. Similarly, we also have
\begin{equation}\label{005}
|B^{\rm s}_i(\delta)|\ge\ubar{\lambda}|\widehat{B}^{\rm s}(\delta)|>\ubar{\lambda}\sigma_0^{-1}|\widehat{B}^{\rm u}(\delta)|>\ubar{\lambda}(a \ubar{\sigma})^{-1}|\widehat{B}^{\rm u}(\delta)|\ge\ubar{\lambda}a^{-1}|B^{\rm u}_i(\delta)|,
\end{equation}
where $a>0$ is a constant independent of $\varepsilon$ satisfying 
\[
\sigma_0=\overline{\sigma}\left(\dfrac{1+c\varepsilon}{1-c\varepsilon}\right)^2<\dfrac{\overline{\sigma}}{\overline{\sigma}-2}<a\ubar{\sigma},\] see \eqref{001} and \eqref{003}. Therefore, \eqref{004} and \eqref{005} together indicate that $(B^{\rm s}_i(\delta),B^{\rm u}_i(\delta))$ are proportional for $i=1,2$. 

In the following, we will show that $(B^{\rm s}_i(\delta),B^{\rm u}_i(\delta))$ are $\xi_0$-linked for $i=1,2$. 
First, we notice that 
\begin{equation}\label{009}
|B^{\rm s}_i(\delta)|\le \overline{\lambda}|\widehat{B}^{\rm s}(\delta)|<\overline{\sigma}^{-1}|\widehat{B}^{\rm u}(\delta)|\le|B^{\rm u}_i(\delta)|,
\end{equation}
which implies that 
\[\min\{|B^{\rm s}_i(\delta)|,|B^{\rm u}_i(\delta)|\}=|B^{\rm s}_i(\delta)|.\]

Two cases will subsequently arise: (a) $\widehat{G}^{\rm s}(\delta)\subset\widehat{G}^{\rm u}(\delta)$ and (b) $\widehat{G}^{\rm s}(\delta)\supset\widehat{G}^{\rm u}(\delta)$. See Figure \ref{fig4}.
\begin{figure}[hbt]
\centering
\scalebox{1}{
\includegraphics[clip]{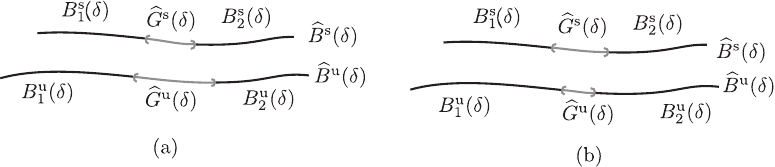}
}
\caption{Two cases in the proof of Lemma~\ref{lem1}} 
\label{fig4}
\end{figure}

In case (a), for $i=1,2$, we have
\begin{align*}
\dfrac{|B^{\rm s}_i(\delta)\cap B^{\rm u}_i(\delta)|}{\min\{|B^{\rm s}_i(\delta)|,|B^{\rm u}_i(\delta)|\}}
&=\dfrac{|B^{\rm s}_i(\delta)\cap B^{\rm u}_i(\delta)|}{|B^{\rm s}_i(\delta)|}>\dfrac{\frac{1}{3}\big(|\widehat{B}^{\rm s}(\delta)|-|\widehat{G}^{\rm u}(\delta)|\big)}{\overline{\lambda}|\widehat{B}^{\rm s}(\delta)|}\\
&>\dfrac{\frac{1}{3}\big(|\widehat{B}^{\rm s}(\delta)|-(1-2\overline{\sigma}^{-1})|\widehat{B}^{\rm u}(\delta)|\big)}{\overline{\sigma}^{-1}|\widehat{B}^{\rm s}(\delta)|}.
\end{align*}
Note that (see item (2) of Claim \ref{clm1}) $$|\widehat{B}^{\rm u}(\delta)|\ge|\widehat{B}^{\rm s}(\delta)|>\sigma_0^{-1}|\widehat{B}^{\rm u}(\delta)|,$$ thus the last term is greater than
\[
\dfrac{\frac{1}{3}\big(\sigma_0^{-1}|\widehat{B}^{\rm u}(\delta)|-(1-2\overline{\sigma}^{-1})|\widehat{B}^{\rm u}(\delta)|\big)}{\overline{\sigma}^{-1}|\widehat{B}^{\rm u}(\delta)|}.\]
Combining this fact with $\sigma_0^{-1}>\dfrac{2}{\overline{\sigma}+3}$ by \eqref{006}, we finally obtain
\[
\dfrac{|B^{\rm s}_i(\delta)\cap B^{\rm u}_i(\delta)|}{\min\{|B^{\rm s}_i(\delta)|,|B^{\rm u}_i(\delta)|\}}
>\dfrac{\sigma_0^{-1}-(1-2\overline{\sigma}^{-1})}{3\overline{\sigma}^{-1}}\\
>\dfrac{(\overline{\sigma}+2)(3-\overline{\sigma})}{3(\overline{\sigma}+3)}=\xi_0.
\]

In case (b), it suffices to show that $B^{\rm s}_i(\delta)$ is not completely contained in any gap of $B^{\rm u}_i(\delta)$ for $i=1,2$. We argue by contradiction. 
Suppose that $B^{\rm s}_1(\delta)$ were 
contained in some gap $\widetilde{G}^{\rm u}(\delta)$ of $B_1^{\rm u}(\delta)$. 
Let $\widetilde{B}^{\rm u}(\delta)$ be one of the adjacent bridges of $\widetilde{G}^{\rm u}(\delta)$. That is, $\widetilde{B}^{\rm u}(\delta)\cap \widetilde{G}^{\rm u}(\delta)\not=\emptyset$ and $\mbox{int}(\widetilde{B}^{\rm u}(\delta))\cap \widetilde{G}^{\rm u}(\delta)=\emptyset$. Thus, 
\[\dfrac{|\widetilde{B}^{\rm u}(\delta)|}{|\widehat{G}^{\rm s}(\delta)|}\cdot \dfrac{|B^{\rm s}_1(\delta)|}{|\widetilde{G}^{\rm u}(\delta)|}<1\cdot 1=1.\]
On the other hand, \eqref{068} gives 
 \[\dfrac{|\widetilde{B}^{\rm u}(\delta)|}{|\widehat{G}^{\rm s}(\delta)|}\cdot \dfrac{|B^{\rm s}_1(\delta)|}{|\widetilde{G}^{\rm u}(\delta)|}=
 \dfrac{|\widetilde{B}^{\rm u}(\delta)|}{|\widetilde{G}^{\rm u}(\delta)|}\cdot \dfrac{|B^{\rm s}_1(\delta)|}{|\widehat{G}^{\rm s}(\delta)|}
   >\tau^{\rm u}_{f_\delta}\tau^{\rm s}_{f_\delta}>1,\]
which gives a contradiction. Similar arguments show that $B^{\rm s}_2(\delta)$ is not completely contained in any gap of $B^{\rm u}_2(\delta)$. Since $B^{\rm s}_i(\delta)\subset B^{\rm u}_i(\delta)$ for $i=1,2$, 
we have
\[\dfrac{|B^{\rm s}_i(\delta)\cap B^{\rm u}_i(\delta)|}{\min\{|B^{\rm s}_i(\delta)|,|B^{\rm u}_i(\delta)|\}}=\dfrac{|B^{\rm s}_i(\delta)|}{|B^{\rm s}_i(\delta)|}=1>\xi_0.\] The proof of Lemma \ref{lem1} is completed now.   
\end{proof}

\begin{proof}[Proof of Claim \ref{clm1}]
(1) By the length estimations in Lemma \ref{lem3} and the choice of $\delta$ under \eqref{069}, we have
\[|\widehat{B}^{\rm s}(\delta)|\le(1+c\varepsilon)|\widehat{B}^{\rm s}(0)|<(1+c\varepsilon)\lambda_0\varepsilon/2=\ubar{\lambda}\varepsilon/2,\]
where the second inequality follows from the choice of $\widehat{b}_{\rm s}$ and the last equality follows from \eqref{003}. Similarly,
\[|\widehat{B}^{\rm s}(\delta)|\ge(1-c\varepsilon)|\widehat{B}^{\rm s}(0)|>(1-c\varepsilon)\lambda^2_0\varepsilon/2>\ubar{\lambda}^3\varepsilon/2,\]
where the last inequality follows from \eqref{001} and \eqref{003}.


(2) In the same manner as in the proof of item (1), we have
\begin{align*}
|\widehat{B}^{\rm u}(\delta)|&\ge(1-c\varepsilon)|\widehat{B}^{\rm u}(0)|\ge(1+c\varepsilon)|\widehat{B}^{\rm s}(0)|\ge|\widehat{B}^{\rm s}(\delta)|,\\
|\widehat{B}^{\rm u}(\delta)|&\le(1+c\varepsilon)|\widehat{B}^{\rm u}(0)|<(1-c\varepsilon)\sigma_0|\widehat{B}^{\rm s}(0)|\le\sigma_0|\widehat{B}^{\rm s}(\delta)|.
\end{align*}

(3) We also 
have
\begin{align*}
|\widehat{B}^{\rm s}(\delta)|
&>\sigma_0^{-1}|\widehat{B}^{\rm u}(\delta)|>(1-2\overline{\sigma}^{-1})|\widehat{B}^{\rm u}(\delta)|\\
&\ge|\widehat{B}^{\rm u}(\delta)|-\Big(|\widehat{B}^{\rm u}_1(\delta)|+|\widehat{B}^{\rm u}_2(\delta)|\Big)=|\widehat{G}^{\rm u}(\delta)|.
\end{align*} 
Here, the first inequality follows from item (2). To obtain the second inequality, it is enough to notice that, according to \eqref{001} and \eqref{003}, we have
\[\sigma_0=\overline{\sigma}\left(\dfrac{1+c\varepsilon}{1-c\varepsilon}\right)^2<\dfrac{\overline{\sigma}}{\overline{\sigma}-2}.\] Now, we complete the proof of Claim \ref{clm1}.
\end{proof}

\begin{remark}\label{010}
In the proof of Lemma \ref{lem1}, we see that the size $\varepsilon$ of the perturbation can be designated in advance as small as we want. Once this $\varepsilon$ is fixed, then the sizes of the sub-bridges $|B_{1,2}^{\rm s}(\delta)|$ and $|B_{1,2}^{\rm u}(\delta)|$ are of order $\varepsilon$. More precisely, according to Claim \ref{clm1}, we have
\begin{gather}
\ubar{\lambda}^4\varepsilon/2<\ubar{\lambda}|\widehat{B}^{\rm s}(\delta)|\le|B_{1,2}^{\rm s}(\delta)|\le\overline{\lambda}|\widehat{B}^{\rm s}(\delta)|<\overline{\lambda}^2\varepsilon/2,\\
|B^{\rm s}_{1,2}(\delta)|<|B^{\rm u}_{1,2}(\delta)|\le\dfrac{1}{\ubar{\sigma}}|\widehat{B}^{\rm u}(\delta)|<\dfrac{\sigma_0}{\ubar{\sigma}}|\widehat{B}^{\rm s}(\delta)|<\ubar{\lambda}\dfrac{\varepsilon\sigma_0}{2\ubar{\sigma}}<a\ubar{\lambda}\dfrac{\varepsilon}{2},
\end{gather}
where $a>0$ is the constant defined under \eqref{005}, which is independent of $\varepsilon$ and satisfies $\sigma_0<a\ubar{\sigma}$.
\end{remark}

\section{Linear Growth Lemma}\label{063}

Let $f$ be an arbitrary element of $\mathcal{U}^r_F$. The objective of this section is to prove the following lemma, which generalizes Linear Growth Lemma in \cite{CV01} to all elements in $\mathcal{U}^r_F$ and establishes the $C^r$-robustness of its conclusions. In particular, this lemma allows us to obtain a sequence of linked pairs of s-bridge and u-bridge by an arbitrarily small slid perturbation such that their generations have linear growth.

Recall that $\xi_0\in (0,1)$ is the constant defined in \eqref{xi}. 

\begin{lem}[Linear Growth Lemma]\label{lem2}
There exist positive constants \[N_{\rm s}=N_{\rm s}(\overline{\lambda},\ubar{\lambda},\overline{\sigma},\ubar{\sigma},\kappa) \quad \text{and}\quad N_{\rm u}=N_{\rm u}(\overline{\lambda},\ubar{\lambda},\overline{\sigma},\ubar{\sigma},\kappa)\] satisfying the following: Let $(B^{\rm s}(0),B^{\rm u}(0))$ be a linked pair of $f$. For every $\varepsilon>0$, there exist $\Delta\in(-\varepsilon, \varepsilon)$, sequences $\{B_k^{\rm s}(\Delta)\}_{k\in \mathbb{N}}$ and $\{B_k^{\rm u}(\Delta)\}_{k\in \mathbb{N}}$ of sub-bridges of $B^{\rm s}(\Delta)$ and $B^{\rm u}(\Delta)$ respectively, such that for every $k\in \mathbb{N}$, the followings hold:
\begin{itemize}
\item[(1)] The pair $(B^{\rm s}_k(\Delta),B^{\rm u}_k(\Delta))$ is $\xi_0/2$-linked. Here, $B_{k}^{{\rm s},{\rm u}}(\Delta)$ is the ${\rm s,u}$-bridges on $L(\Delta)$ with respect to $f_{\Delta}$.
\item[(2)] Let $s_k, u_k$ be the generations of $B_k^{\rm s}, B_k^{\rm u}$, then
\[s_{k+1}-s_k\le N_{\rm s} \quad \text{and}\quad u_{k+1}-u_k\le N_{\rm u}.\]
\end{itemize} 
\end{lem}

\begin{proof}
Fix an arbitrarily small $\varepsilon>0$. The following proof will be divided into four steps. The two sequences in the statement will be obtained in Step 3 and item (1) is proved at the end. Item (2) is proved in Step 4.

\vspace{0.2cm}
\noindent {\it Step 1.} In this step, we will prove the following claim.
\begin{clm}\label{clm2}
For every $k\in \mathbb{N}$, there exist $\Delta_k\in \mathbb{R}$ and $\xi_0$-linked pairs $(B^{\rm s}_t(\Delta_k),B^{\rm u}_t(\Delta_k))$ $(t=1,\dots,k)$ in $L(\Delta_k)$ such that $B^{\rm s}_t(\Delta_k)$ $(t=1,\dots,k)$ are pairwise disjoint and $B^{\rm u}_t(\Delta_k)$ $(t=1,\dots,k)$ are pairwise disjoint.
\end{clm}

\begin{proof}[Proof of Claim]
We will construct this sequence of linked pairs by induction. First, for $k=1$, let us find a $\xi_0$-linked pair $(B^{\rm s}_t(\Delta_1),B^{\rm u}_t(\Delta_1))$ in $L(\Delta_1)$. Indeed, since $B^{\rm s}(0)$ and $B^{\rm u}(0)$ are linked, we are allowed to apply Lemma \ref{lem1} to this pair. For $\varepsilon>0$ fixed before, there are $\delta_1$ with $|\delta_1|<\varepsilon/2$, sub-bridges $B_1^{\rm s}(\delta_1), \widetilde{B}_1^{\rm s}(\delta_1)$ of $B^{\rm s}(\delta_1)$, and $B_1^{\rm u}(\delta_1), \widetilde{B}_1^{\rm u}(\delta_1)$ of $B^{\rm u}(\delta_1)$ such that 
\[(B^{\rm s}_1(\delta_1),B^{\rm u}_1(\delta_1))\quad \text{and} \quad (\widetilde{B}^{\rm s}_1(\delta_1),\widetilde{B}^{\rm u}_1(\delta_1))\] are $\xi_0$-linked pairs. Let $\Delta_1:=\delta_1$, then $(B^{\rm s}_1(\Delta_1), B^{\rm u}_1(\Delta_1))$ is exactly the first pair of the sequence that we desired in the statement of Claim \ref{clm2}, while the other pair $(\widetilde{B}^{\rm s}_1(\Delta_1),\widetilde{B}^{\rm u}_1(\Delta_1))$ will be used for constructing the next pair.
Moreover, by Remark \ref{010}, we also have the length estimation 
\begin{equation}\label{011}
|B^{\rm s}_1(\Delta_1)|< \overline{\lambda}^2\varepsilon/2.
\end{equation}

We set $\Delta_k=\delta_1+\cdots +\delta_k$. Suppose by induction that for some $k\ge 1$, we have found $\xi_0$-linked pairs $(B^{\rm s}_k(\Delta_k),B^{\rm u}_k(\Delta_k))$ and $(\widetilde{B}^{\rm s}_k(\Delta_k),\widetilde{B}^{\rm u}_k(\Delta_k))$ such that $B^{\rm s}_t(\Delta_k)$ (resp. $B^{\rm u}_t(\Delta_k)$) $(t=1,\dots,k)$ are pairwise disjoint. Now, to prove the claim, it remains to find $\Delta_{k+1}\in \mathbb{R}$ and $\xi_0$-linked pairs $(B^{\rm s}_t(\Delta_{k+1}),B^{\rm u}_t(\Delta_{k+1}))$ $(t=1,\dots,k+1)$ in $L(\Delta_{k+1})$ such that $B^{\rm s}_t(\Delta_{k+1})$ (resp. $B^{\rm u}_t(\Delta_{k+1})$) $(t=1,\dots,k+1)$ are pairwise disjoint.
For this, let $s_k$ and $u_k$ denote the generations of $B^{\rm s}_k(\Delta_k)$ and $B^{\rm u}_k(\Delta_k)$ respectively. Since $B^{\rm s}_k(\Delta_k)$ and $\widetilde{B}^{\rm s}_k(\Delta_k)$ (resp. $B^{\rm u}_k(\Delta_k)$ and $\widetilde{B}^{\rm u}_k(\Delta_k)$) are related bridges obtained by Lemma \ref{lem1}, they have the same generation (see the argument below \eqref{069}). Next, by applying Lemma \ref{lem1} to $(\widetilde{B}^{\rm s}_k(\Delta_k),\widetilde{B}^{\rm u}_k(\Delta_k))$ for 
\begin{equation}\label{034}
\varepsilon_k:=\dfrac{\overline{\lambda}\xi_0}{4(\kappa+1)}|B_k^{\rm s}(\Delta_k)|,
\end{equation} where $\kappa$ is the constant in Lemma \ref{lem3}, there exist $\delta_{k+1}$ with $|\delta_{k+1}|<\varepsilon_k$ and sub-bridges 
\begin{align*}
B_{k+1}^{\rm s}(\Delta_{k}+\delta_{k+1}),\ \widetilde{B}_{k+1}^{\rm s}(\Delta_{k}+\delta_{k+1}) \quad &\text{of} \ \widetilde{B}^{\rm s}_k(\Delta_k+\delta_{k+1}),\\
B_{k+1}^{\rm u}(\Delta_{k}+\delta_{k+1}),\  \widetilde{B}_{k+1}^{\rm u}(\Delta_{k}+\delta_{k+1}) \quad &\text{of}\ \widetilde{B}^{\rm u}_k(\Delta_k+\delta_{k+1})
\end{align*}
of generations $s_{k+1}$ and $u_{k+1}$ respectively such that 
\[
(B^{\rm s}_{k+1}(\Delta_{k+1}),B^{\rm u}_{k+1}(\Delta_{k+1}))\quad \text{and}\quad (\widetilde{B}^{\rm s}_{k+1}(\Delta_{k+1}),\widetilde{B}^{\rm u}_{k+1}(\Delta_{k+1}))\]
are $\xi_0$-linked pairs, 
where we define 
$$\Delta_{k+1}:=\Delta_k+\delta_{k+1}.$$ 
Moreover, we see that according to the above process, sub-bridges $B_t^{\rm s}(\Delta_{k+1})$ and $B_t^{\rm u}(\Delta_{k+1})$ are also well-defined for every $t=1,2,\dots, k$. Indeed, we have
\begin{align*}
B_t^{\rm s}(\Delta_{k+1})&:=B_t^{\rm s}(\Delta_t+\delta_{t+1}+\cdots +\delta_{k+1}),  
\\
B_t^{\rm u}(\Delta_{k+1})&:=B_t^{\rm u}(\Delta_t+\delta_{t+1}+\cdots +\delta_{k+1}).
\end{align*}
Notice that $B^{{\rm s({\rm u})}}_{k+1}(\Delta_{k+1})$ is contained inside $\widetilde{B}^{{\rm s}(\rm u)}_k(\Delta_{k+1})$ which is disjoint from $B^{{\rm s}(\rm u)}_k(\Delta_{k+1})$. Thus, $B^{\rm s}_t(\Delta_{k+1})$ (resp. $B^{\rm u}_t(\Delta_{k+1})$) $(t=1,\dots,k+1)$ are pairwise disjoint. 

In addition, by applying Lemma \ref{lem3} finite many times, we have the following estimation which will be useful in Step 3.
\begin{equation}\label{014}
\begin{split}
|B^{\rm s}_{k+1}(\Delta_{k+1})|
&\le \overline{\lambda}^{s_{k+1}-s_k}|B^{\rm s}_{k}(\Delta_{k+1})|\\
&\le \overline{\lambda}^{s_{k+1}-s_k}|B^{\rm s}_{k}(\Delta_{k})|(1+c\delta_{k+1})\\
&\le \cdots \le \overline{\lambda}^{s_{k+1}-s_1}|B^{\rm s}_1(\Delta_1)|\prod_{i=2}^{k+1}(1+c\delta_{i}).
\end{split}
\end{equation}
This completes the proof of Claim \ref{clm2}. 
\end{proof}

\vspace{0.2cm}
\noindent {\it Step 2.}
For every $k\in\mathbb{N}$, let us denote \[\xi_k:=\xi_0\left(1-\dfrac{1}{2}\sum_{i=1}^k\overline{\lambda}^i\right).\] It is clear that $\xi_k>\xi_0/2$ for every $k\in\mathbb{N}$. In this step, we will prove the following claim. 
\begin{clm}\label{clm3}
For every $k\in \mathbb{N}$, the pair $(B^{\rm s}_t(\Delta_k),B^{\rm u}_t(\Delta_k))$ are $\xi_{k-t}$-linked for every $t=1,2,\dots,k$.
\end{clm}

\begin{proof}[Proof of Claim]
The proof of Claim \ref{clm3} will be given by induction on $k$. When $k=1$, the only case we need to consider is $t=1$. The conclusion follows directly from Claim \ref{clm2}. Suppose the conclusion of Claim \ref{clm3} holds for $k$, in the following, we will show that $(B^{\rm s}_t(\Delta_{k+1}),B^{\rm u}_t(\Delta_{k+1}))$ are $\xi_{k+1-t}$-linked for every $t=1,2,\dots,k+1$. 

When $t=k+1$, by the construction of $(B^{\rm s}_{k+1}(\Delta_{k+1}), B^{\rm u}_{k+1}(\Delta_{k+1}))$ in Step 1, we are done, since this pair is $\xi_0$-linked. Thus, it suffices to set $t\in \{1,2,\dots,k\}$. Then, we have
\begin{align*}
|B^{\rm s}_t(\Delta_{k+1})\cap B^{\rm u}_t(\Delta_{k+1})|
&\ge (1-c\delta_{k+1})|B^{\rm s}_t(\Delta_{k})\cap B^{\rm u}_t(\Delta_{k})|-\kappa\delta_{k+1}\\
&\ge (1-c\delta_{k+1}) \xi_{k-t}|B^{\rm s}_t(\Delta_k)|-\kappa\delta_{k+1}\\
&\ge (1-c\delta_{k+1}) \xi_{k-t}\dfrac{|B^{\rm s}_t(\Delta_{k+1})|}{1+c\delta_{k+1}}-\kappa\delta_{k+1}\\
&=\xi_{k-t}|B^{\rm s}_t(\Delta_{k+1})|-\delta_{k+1}\big(\kappa+2c\xi_{k-t}|B^{\rm s}_t(\Delta_{k+1})|+O(\delta_{k+1})\big).
\end{align*} 
Here, we have applied Lemma \ref{lem3} in the first and third inequalities, while the second inequality is obtained according to the induction hypothesis. If $\varepsilon$ is fixed sufficiently small in advance, then we are allowed to bound the coefficient of $\delta_{k+1}$ in the last line from above by $(\kappa+1)$. Indeed, \eqref{034} gives that
$|\delta_{k+1}|<\varepsilon_k=O(|B^{\rm s}_k(\Delta_k)|)$.
Moreover, both of $|B^{\rm s}_t(\Delta_{k+1})|$ and $|B^{\rm s}_k(\Delta_k)|$ are bounded from above by $|B^{\rm s }_1(\Delta_1)|=O(\varepsilon)$ as $\varepsilon$ tends to zero. Thus,
\[\kappa+2c\xi_{k-t}|B^{\rm s}_t(\Delta_{k+1})|+O(\delta_{k+1})<\kappa+1\] for every $\varepsilon>0$ small enough.

Hence it follows that
\begin{equation}\label{008}
\begin{split}
|B^{\rm s}_t(\Delta_{k+1})\cap B^{\rm u}_t(\Delta_{k+1})|
&\ge \xi_{k-t}|B^{\rm s}_t(\Delta_{k+1})|-(\kappa+1)\delta_{k+1}\\
&=\left(\xi_{k-t}-\dfrac{(\kappa+1)\delta_{k+1}}{|B^{\rm s}_t(\Delta_{k+1})|}\right) |B^{\rm s}_t(\Delta_{k+1})|.
\end{split}
\end{equation} 
By bounded distortion property (see Lemma \ref{lem4} and Remark \ref{139}), we have
\[\dfrac{|B^{\rm s}_k(\Delta_{k+1})|}{|B^{\rm s}_{t}(\Delta_{k+1})|}\le\overline{\lambda}^{s_k-s_{t}},\] because $s_k$ is the generation of $B^{\rm s}_k(\Delta_{k+1})$. Thus, by recalling \eqref{034} and noticing that $1-c|\delta_{k+1}|>1/2$, we have 
\begin{equation}\label{007}
\begin{split}
\dfrac{(\kappa+1)\delta_{k+1}}{|B^{\rm s}_t(\Delta_{k+1})|}
&\le (\kappa+1) \dfrac{\overline{\lambda}\xi_0}{4(\kappa+1)} \dfrac{|B^{\rm s}_k(\Delta_{k})|}{|B^{\rm s}_t(\Delta_{k+1})|} 
\le \dfrac{\bar{\lambda}\xi_0}{4(1-c|\delta_{k+1}|)} \dfrac{|B^{\rm s}_k(\Delta_{k})|}{|B^{\rm s}_t(\Delta_{k})|} \\
&\le \dfrac{\overline{\lambda}\xi_0}{2}\overline{\lambda}^{s_k-s_{t}}\le \dfrac{\overline{\lambda}\xi_0}{2} \overline{\lambda}^{k-t}=\dfrac{\xi_0}{2}\overline{\lambda}^{k+1-t}.
\end{split}
\end{equation} 
In the last inequality, we used the obvious relation that 
\[s_k-s_t\ge k-t.\]
Therefore, by substituting \eqref{007} into \eqref{008}, and by recalling that $|B^{\rm s}_t(\Delta_{k+1})|\le|B^{\rm u}_t(\Delta_{k+1})|$ according to the proof of Lemma \ref{lem1} (see \eqref{009}), we get
\begin{align*}
\dfrac{|B^{\rm s}_t(\Delta_{k+1})\cap B^{\rm u}_t(\Delta_{k+1})|}{|B^{\rm s}_t(\Delta_{k+1})|}
&\ge\xi_{k-t}-\dfrac{\xi_0}{2}\overline{\lambda}^{k+1-t}\\
&=\xi_0\left(1-\dfrac{1}{2}\sum_{i=1}^{k-t}\overline{\lambda}^i\right)-\dfrac{\xi_0}{2}\overline{\lambda}^{k+1-t}\\
&=\xi_{k+1-t}.
\end{align*}
This completes the proof of Claim \ref{clm3}.
\end{proof}

\vspace{0.2cm}
\noindent {\it Step 3.} In this step, we will finish the construction of $(B^{\rm s}_k(\Delta),B^{\rm u}_k(\Delta))_{k\in \mathbb{N}}$ for a uniform constant $\Delta$ which is independent of $k$.

For every integer $k\ge 2$, by recalling the construction of $(B^{\rm s}_t(\Delta_k), B^{\rm u}_k(\Delta_k))$ $(t=1,\dots,k)$ in Step 1, we have the following estimation, see \eqref{014},
\begin{equation}\label{019}
\begin{split}
|B^{\rm s}_k(\Delta_k)|
&\le \overline{\lambda}^{s_k-s_{k-1}}|B^{\rm s}_{k-1}(\Delta_k)|
\le  \overline{\lambda}^{s_k-s_{k-1}}(1+c\delta_k)|B^{\rm s}_{k-1}(\Delta_{k-1})|\\
&\le\cdots \le \overline{\lambda}^{s_k-s_{1}}|B^{\rm s}_{1}(\Delta_{1})|\prod_{i=2}^k(1+c\delta_i)\\
&\le \overline{\lambda}^{k-1}|B^{\rm s}_1(\Delta_1)|\prod_{i=2}^k(1+c\delta_i)
= |B^{\rm s}_1(\Delta_1)|\prod_{i=2}^k [\overline{\lambda}(1+c\delta_i)]\\
&\le |B^{\rm s}_1(\Delta_1)|\left(\dfrac{3}{4}\right)^{k-1},
\end{split}
\end{equation}
where the last inequality holds since $\overline{\lambda}<1/2$ and $1+c\delta_i<3/2$ for a small $\varepsilon$. Notice that \eqref{019} holds for $k=1$ as well.

Now, we are in the position to define
\[\Delta=\lim\limits_{k\to\infty}\Delta_k=\sum_{k=1}^\infty\delta_k.\]
By substituting \eqref{019} into 
\begin{equation}\label{012}
|\delta_{k+1}|
\le \varepsilon_k= \dfrac{\overline{\lambda} \xi_0}{4(\kappa+1)} |B^{\rm s}_k(\Delta_k)|
\end{equation}
and combining it with $\bar{\lambda}<1/2<3/4$ and $\kappa>2$, we obtain
\begin{equation}\label{186}
\begin{split}
|\Delta| \le \sum_{k=1}^\infty|\delta_k|
&=|\delta_1|+\sum_{k=1}^\infty|\delta_{k+1}|\\
&\le \dfrac{\varepsilon}{2}+\sum_{k=1}^\infty\dfrac{\overline{\lambda}\xi_0}{4(\kappa+1)}|B^{\rm s}_1(\Delta_1)|\left(\dfrac{3}{4}\right)^{k-1} \\
&< \dfrac{\varepsilon}{2}+\dfrac{\xi_0}{4(\kappa+1)} |B^{\rm s}_1(\Delta_1)|\sum_{k=1}^\infty\left(\dfrac{3}{4}\right)^{k} \\
&\le  \dfrac{\varepsilon}{2}+ \dfrac{\xi_0}{4}|B^{\rm s}_1(\Delta_1)|\le \varepsilon.
\end{split}
\end{equation}
Here, we have applied \eqref{011} in the last inequality. The above estimation \eqref{186} shows that $\{B^{\rm s}_k(\Delta_l)\}_{l\in\mathbb{N}}$ and $\{B^{\rm u}_k(\Delta_l)\}_{l\in\mathbb{N}}$ are Cauchy sequences of compact sets 
with respect to the Hausdorff metric. Let us explain the reason for $\{B^{\rm u}_k(\Delta_l)\}_{l\in\mathbb{N}}$ and the same reason holds for $\{B^{\rm s}_k(\Delta_l)\}_{l\in\mathbb{N}}$. For every $l\in\mathbb{N}$, we denote the left and right endpoints of $B^{\rm u}_k(\Delta_l)$ by $\boldsymbol{a}_l$ and $\boldsymbol{b}_l$, respectively. For the proof, it suffices to show that $\{\boldsymbol{a}_l\}_{l\in\mathbb{N}}$ and $\{\boldsymbol{b}_l\}_{l\in\mathbb{N}}$ are Cauchy sequences. Note that
\[B^{\rm u}_k(\Delta_l)\subset L(\Delta_l) \quad (l\in\mathbb{N})\quad\mbox{and} \quad L(\Delta_l)\to L(\Delta) \quad (l\to \infty).\]
There exists a constant $C$ independent of $k$ and $l$ such that 
\[\mbox{dist}(\boldsymbol{a}_N,\boldsymbol{a}_{N+1})\le C|\Delta_{N+1}-\Delta_{N}|=C|\delta_{N+1}|.\]
It follows that
\[\mbox{dist}(\boldsymbol{a}_N,\boldsymbol{a}_{N+p})
\le\sum_{i=0}^{p-1}\mbox{dist}(\boldsymbol{a}_{N+i},\boldsymbol{a}_{N+i+1})
\le C\sum_{i=1}^{p}|\delta_{N+i}|.\]
Because the series $\sum_{k=1}^\infty|\delta_k|$ converges by \eqref{186}, for an arbitrarily small $\eta>0$, there exists $N\in\mathbb{N}$ large enough such that $\mbox{dist}(\boldsymbol{a}_N,\boldsymbol{a}_{N+p})<\eta$ for every $p\in\mathbb{N}$. Similar argument can be applied to show that $\{\boldsymbol{b}_l\}_{l\in\mathbb{N}}$ is a Cauchy sequence as well.

Thus, we are allowed to define, for every $k\in\mathbb{N}$, that
\[B_k^{\rm s}(\Delta):=\lim\limits_{l\to \infty}B^{\rm s}_k(\Delta_l) \quad \text{and}\quad B_k^{\rm u}(\Delta):=\lim\limits_{l\to \infty}B^{\rm u}_k(\Delta_l).\]
Note that $\xi_k$ has a uniform lower bound $\xi_0/2$ as we mentioned in the beginning of Step 2. By taking the limit, Claim \ref{clm3} implies that $(B^{\rm s}_k(\Delta),B^{\rm u}_k(\Delta))$ is $\xi_0/2$-linked for every $k\in\mathbb{N}$ as we desired in Lemma \ref{lem2} (1).

\vspace{0.2cm}
\noindent {\it Step 4.} 
Finally, let us show (2). Since $B^{\rm s}_{k+1}(\Delta_{k+1})$ is obtained by applying Lemma \ref{lem1} to $(\widetilde{B}^{\rm s}_k(\Delta_k),\widetilde{B}^{\rm u}_k(\Delta_k))$ for $\varepsilon_k$ (see \eqref{034})
, hence Remark \ref{010} gives
\begin{align}\label{015}
 |B^{\rm s}_{k+1}(\Delta_{k+1})|&\ge \ubar{\lambda}^4 \varepsilon_k/2
= \ubar{\lambda}^4 \dfrac{\overline{\lambda}\xi_0}{8(\kappa+1)} |B^{\rm s}_{k}(\Delta_k)|
\ge \dfrac{\ubar{\lambda}^5\xi_0}{8(\kappa+1)} |B^{\rm s}_{k}(\Delta_k)|\nonumber\\
&\ge \dfrac{\ubar{\lambda}^5\xi_0}{8(\kappa+1)(1+c\delta_{k+1})} |B^{\rm s}_{k}(\Delta_{k+1})|\ge \dfrac{\ubar{\lambda}^5\xi_0}{12(\kappa+1)} |B^{\rm s}_{k}(\Delta_{k+1})|,
\end{align}
where in the last inequality, we have used the estimation $1+c\delta_{k+1}<3/2$.
It follows that (refer to Lemma \ref{lem4} and Remark \ref{139})
\[\overline{\lambda}^{s_{k+1}-s_k}\ge \dfrac{|B^{\rm s}_{k+1}(\Delta_{k+1})|}{|B^{\rm s}_{k}(\Delta_{k+1})|}\ge\dfrac{\ubar{\lambda}^5\xi_0}{12(\kappa+1)},\]
which gives
\[s_{k+1}-s_k\le (\log\overline{\lambda})^{-1} \log \dfrac{\ubar{\lambda}^5\xi_0}{12(\kappa+1)} 
=:N_{\rm s}\]
as desired.

Next, we consider the case of u-bridges. Recall that $(B^{\rm s}_{k+1}(\Delta_{k+1}), B^{\rm u}_{k+1}(\Delta_{k+1}))$ 
is the proportional pair obtained by using Lemma \ref{lem1}. It follows that
\begin{align*}
|B^{\rm u}_{k+1}(\Delta_{k+1})|
&\ge  |B^{\rm s}_{k+1}(\Delta_{k+1})|
\ge \dfrac{\ubar{\lambda}^5\xi_0}{12(\kappa+1)} |B^{\rm s}_{k}(\Delta_{k+1})|\\
&\ge \dfrac{\ubar{\lambda}^5\xi_0}{12(\kappa+1)} (1-c\delta_{k+1})|B^{\rm s}_{k}(\Delta_{k})|\\
&\ge \dfrac{\ubar{\lambda}^5\xi_0}{12(\kappa+1)} (1-c\delta_{k+1}) \ubar{\lambda}a^{-1} |B^{\rm u}_{k}(\Delta_{k})|\\
&\ge \dfrac{\ubar{\lambda}^5\xi_0}{12(\kappa+1)} (1-c\delta_{k+1}) \ubar{\lambda}a^{-1}  \dfrac{|B^{\rm u}_{k}(\Delta_{k+1})|}{(1+c\delta_{k+1})},
\end{align*}
where the first and fourth inequalities follow from \eqref{004} and \eqref{005}, the third and fifth inequalities follow from Lemma \ref{lem3}, and the second inequality is given by \eqref{015}. Therefore, we have the following estimations for the u-bridges as well.
\[(\ubar{\sigma}^{-1})^{u_{k+1}-u_k}\ge \dfrac{|B^{\rm u}_{k+1}(\Delta_{k+1})|}{|B^{\rm u}_{k}(\Delta_{k+1})|}
\ge \dfrac{\ubar{\lambda}^6\xi_0(1-c\delta_{k+1})}{12a(\kappa+1)(1+c\delta_{k+1})}\ge \dfrac{\ubar{\lambda}^6\xi_0}{36a(\kappa+1)}, \]
which gives
\[u_{k+1}-u_k\le (-\log\ubar{\sigma})^{-1}\log \dfrac{\ubar{\lambda}^6\xi_0}{36a(\kappa+1)}=:N_{\rm u}.\]
The proof of Lemma \ref{lem2} is completed now.
\end{proof}

\section{Critical Chain Lemma}\label{064}
The primary goal of this section is to construct an infinite sequence called the critical chain, where each member will serve as a positional marker along the forward orbit of the eventually constructed wandering domain. Before formally giving the construction of the critical chain in Subsection~\ref{ccl}, we need some preparation.

Suppose that $\{a_k\}$ and $\{b_k\}$ are two sequences of positive numbers. We introduce the following notations:
\begin{itemize}
\item $a_k\lesssim b_k$ means that there exists some positive constant $K_1$ independent of $k$ such that $a_k\le K_1 b_k$ for every $k$;
\item $a_k\gtrsim b_k$ means that there exists some positive constant $K_2$ independent of $k$ such that $a_k\ge K_2 b_k$ for every $k$;
\item $a_k\sim b_k$ means that $a_k\lesssim b_k$ and $a_k\gtrsim b_k$. In other words, there exist positive constants $K_1,\ K_2$ such that $K_2\le a_k/b_k\le K_1$ for every $k$.
\end{itemize}

For a sequence of closed intervals $[a_k,b_k]$ on $\mathbb{R}$ and $\rho>0$, we say that $[a_k,b_k]$ are {\it $\rho$-uniformly pairwise disjoint} if $$\left[a_k-\rho(b_k-a_k),\ b_k+\rho(b_k-a_k)\right]$$ are pairwise disjoint for all $k$. 
We say that $[a_k,b_k]$ are {\it uniformly pairwise disjoint} if $[a_k,b_k]$ are $\rho$-uniformly pairwise disjoint for some $\rho>0$. Similar definitions can also be given for sequences of intervals on $C^1$ arcs.

Let $\phi$ be a non-decreasing $C^\infty$ function defined on $\mathbb{R}$ satisfying 
\begin{equation}
\phi(x)=\left\{
\begin{aligned}
&0\quad \mbox{if}\ x\le-1,\\
&1\quad \mbox{if}\ x\ge0.\\
\end{aligned}
\right.
\end{equation}
Given $\rho>0$ and an interval $[a,b]$, let
\[\phi_{\rho,[a,b]}(x):=\phi\left(\dfrac{x-a}{\rho(b-a)}\right)+\phi\left(\dfrac{b-x}{\rho(b-a)}\right)-1.\]
Thus, $\phi_{\rho,[a,b]}$ is a non-negative $C^\infty$ function on $\mathbb{R}$ satisfying
\begin{itemize}
\item $\mbox{supp} (\phi_{\rho,[a,b]})\subset [a-\rho(b-a),\ b+\rho(b-a)]$,
\item $\phi_{\rho,[a,b]}(x)=1$ for every $x\in[a,b]$,
\item $\phi_{\rho,[a,b]}(x)\in[0,1]$ for every $x\in\mathbb{R}$,
\item $\|\phi_{\rho,[a,b]}\|_{C^r}\le(\rho(b-a))^{-r}\|\phi\|_{C^r}$ if $\rho(b-a)\le 1$.
\end{itemize}
Bump functions of this type will be used later for constructing $C^r$-perturbations.

\subsection{An open subset $\mathcal{U}_0$ of $\mathcal{U}^r_F$}\label{183}
In this subsection, we will select an element $F_0$ and a neighborhood $\mathcal{U}_0\subset \mathcal{U}_F^r$ of $F_0$. They are exactly the diffeomorphism and the open set in the statement of Theorem \ref{mainthm}. Let us Recall that $\mathcal{U}_F^r$ is the small neighborhood of $F$  fixed in Section \ref{060}. For every $f\in \mathcal{U}^r_F$, the tangency curve between $\mathcal{F}^{\rm s}$ and $f^2(\mathcal{F}^{\rm u})$ is $L_f$.  
\begin{clm}\label{140}
There exist an element $F_{0}$ of $\mathcal{U}^r_F$  and 
a neighborhood $\mathcal{U}_0\subset \mathcal{U}^r_F$ of $F_{0}$ 
satisfying the following conditions: For every $f\in \mathcal{U}_0$, there is a linked pair $(B^{\rm s}, B^{\rm u})$ on its tangency curve $L_f$.
\end{clm}
\begin{proof}
First, let us consider the center diffeomorphism $F$ described in Subsection \ref{112}. One easily sees that its tangency curve $L_F$ lies exactly on the $x$-axis $\{y=0\}$. Recall that $p_F=(-a_{\rm u},-a_{\rm s})$ is one of the fixed points of $F$ in Subsection~\ref{112} and $\mu$, $\alpha$, $\beta$, $\gamma$ are the constants associated to $F^2$ in \eqref{f2}. Let us define
\[
n_0:=\min\big\{n\in \mathbb{N}: \ 2a_{\rm u}\sigma^{-(n+1)}\le \mu\big\},\ 
m_0:=\min\big\{m\in \mathbb{N}: \ 2a_{\rm s}\gamma\lambda^{m+1}\le \mu\big\}.
\]
Thus, $B^{\rm u}_0:=B^{\rm u}(n_0; 00\dots 0)\subset\{y=0\}$ has its left boundary at $-a_{\rm u}$ and the right boundary at the right of $-a_{\rm u}+\mu$. Similarly, $B^{\rm s}_0:=B^{\rm s}(m_0; 00\dots 0)\subset\{y=0\}$ has its right boundary at $-a_{\rm u}+\mu$ and the left boundary at the left of $-a_{\rm u}$. It then follows that $|B^{\rm s}_0\cap B^{\rm u}_0|=\mu$ and neither $B^{\rm s}_0$ is contained in the interior of any gap of $B^{\rm u}_0$ nor $B^{\rm u}_0$ is contained in the interior of any gap of $B^{\rm s}_0$. 
Since $\tau^{\rm s}\tau^{\rm u}>1$ by \eqref{thick1}, 
according to Lemma \ref{gaplemma} (Gap Lemma),
we are allowed to find a point $\bm x$ in $(\Lambda^{\rm s}_{L_F}\cap B^{\rm s}_0)\cap(\Lambda^{\rm u}_{L_F}\cap B^{\rm u}_0)$ where $\Lambda^{\rm s}_{L_F}$ and $\Lambda^{\rm u}_{L_F}$ are Cantor sets on $L_F$ defined by \eqref{138}. Note that both $\Lambda^{\rm s}_{L_F}$ and $\Lambda^{\rm u}_{L_F}$ are Cantor sets hence $\bm x$ is not an isolated point of them. Therefore, by perturbing $\mu$ of \eqref{f2} a little bit, precisely, by considering $\mu+c$ instead of $\mu$ for some $c$ with $|c|$ very small, we obtain 
the $c$-slid perturbation of $F$, denoted by $F_0$, such that there exist sub-bridges $B_{F_0}^{\rm s}\subset \mbox{Int}(L_{F_0})$ of $B^{\rm s}_0$ and $B_{F_0}^{\rm u}\subset \mbox{Int}(L_{F_0})$ of $B^{\rm u}_0$ around $\bm x$, satisfying
\begin{itemize}
\item[(i)] $B_{F_0}^{\rm s}$ is not contained in the interior of any gap of $B_{F_0}^{\rm u}$,
\item[(ii)] $B_{F_0}^{\rm u}$ is not contained in the interior of any gap of $B_{F_0}^{\rm s}$,
\item[(iii)] $|B_{F_0}^{\rm s}\cap B_{F_0}^{\rm u}|>0$.
\end{itemize}      
Suppose 
\[
B_{F_0}^{\rm s}=B_{F_0}^{\rm s}(s;\underline{w}), \quad 
B_{F_0}^{\rm u}=B_{F_0}^{\rm u}(u;\underline{z}) 
\]
for some $\underline{w}\in \{0,1\}^{s}$ and $\underline{z}\in \{0,1\}^{u}$. Here, we add the subscript $F_0$ in the notations in order to emphasize that they are the ${\rm s}({\rm u})$-bridges with respect to $F_0$. 

Now, let us take an arbitrary $f$ in a neighborhood $\mathcal{U}_0\subset \mathcal{U}^r_F$ of $F_0$. As long as $\mathcal{U}_0$ is fixed small enough, the above three conditions (i)-(iii) also hold for bridges 
\[B^{\rm s}:=B^{\rm s}(s;\underline{w})\subset L_f, \quad 
B^{\rm u}:=B^{\rm u}(u;\underline{z})\subset L_f\] of every
$f\in \mathcal{U}_0$. We conclude that these bridges are the desired linked pair for $f$, which completes the proof of the claim.
\end{proof}
\begin{remark}\label{rmk7.2}
According to the proof of Claim \ref{140}, it is not hard to see that $\mathcal{U}_0$ can be chosen arbitrarily close to $F$. Indeed, it suffices to select $c$ with $|c|$ sufficiently small in the proof.
\end{remark}
As a consequence of this claim, for every $f\in \mathcal{U}_0$, Lemma \ref{lem2} (Linear Growth Lemma) can be applied to $f$ and this linked pair. Let $\Delta=\Delta(\varepsilon)$ be the constant obtained by Lemma \ref{lem2}. By selecting $\varepsilon>0$ sufficiently small in advance, we can certainly require that $f_{\Delta}$ is still contained in $\mathcal{U}_0$. For notational simplicity, from now on, let us denote the $\Delta$-slid perturbation $f_{\Delta}$ of $f$ by $f$ again. 

Since $f$ satisfies the conclusion of Lemma \ref{lem2}, there exists a sequence of $\xi_0/2$-linked pairs $(B^{\rm s}_k,B^{\rm u}_k)$ with generations $(s_k,u_k)$ respectively. In particular, $s_k$ and $u_k$ satisfy Lemma \ref{lem2} (2) for the constants $N_{\rm s}$ and $N_{\rm u}$. Let us fix a large number $N$. In particular, we assume that $N$ is much larger than $\max\{N_{\rm s}, N_{\rm u}\}$. For every $k\in \mathbb{N}$, since $(B^{\rm s}_k, B^{\rm u}_k)$ is a linked pair, by applying Lemma \ref{gaplemma} (Gap Lemma) to $\Lambda^{\rm s}_{L}\cap B^{\rm s}_k$ and $\Lambda^{\rm u}_{L}\cap B^{\rm u}_k$, there exist linked sub-bridges 
\begin{equation}\label{070}
\widehat{B}^{\rm s}_k:=\widehat{B}^{\rm s}_k(\widehat{s}_k;\widehat{\ubar{w}}^{(k)})\subset B^{\rm s}_k, \quad 
\widehat{B}^{\rm u}_k:=\widehat{B}^{\rm u}_k(\widehat{u}_k;\widehat{\ubar{z}}^{(k)})\subset B^{\rm u}_k  
\end{equation}
with $\widehat{\ubar{w}}^{(k)}\in \{0,1\}^{\widehat{s}_k}$ and $\widehat{\ubar{z}}^{(k)}\in \{0,1\}^{\widehat{u}_k}$,
whose lengths satisfy 
\begin{equation}\label{022}
\ubar{\lambda}^2\cdot\ubar{\lambda}^{kN}\le |\widehat{B}^{\rm s}_k|\le \ubar{\lambda}\cdot\ubar{\lambda}^{kN},\quad 
|\widehat{B}^{\rm s}_k|\le |\widehat{B}^{\rm u}_k|\le \overline{\sigma}|\widehat{B}^{\rm s}_k|.
\end{equation}
It is easy to see that $\widehat{s}_k,\widehat{u}_k\to \infty$ when $k$ tends to infinity. Now, for an arbitrary $\widehat{m}_k\in \mathbb{N}$,  let us define a new itinerary
\begin{equation}\label{020}
\ubar{z}^{(k)}:=\widehat{\ubar{z}}^{(k)}\widehat{\ubar{v}}^{(k)}[\widehat{\ubar{w}}^{(k+1)}]^{-1}
\end{equation}
with length 
\begin{equation}\label{108}
n_k=\widehat{u}_k+\widehat{m}_k+\widehat{s}_{k+1},
\end{equation}
where $\widehat{\ubar{v}}^{(k)}$ is an arbitrary element of $\{0,1\}^{\widehat{m}_k}$. Consider the sub-bridges
\begin{equation}\label{021}
\mathscr{B}^{\rm s}_k:=B^{\rm s}(n_{k-1};[\ubar{z}^{(k-1)}]^{-1})\subset \widehat{B}^{\rm s}_k \quad\text{and}\quad \mathscr{B}^{\rm u}_k:=B^{\rm u}(n_{k};\ubar{z}^{(k)})\subset \widehat{B}^{\rm u}_k.
\end{equation}
Here, the non-consistence of subscript in the definition of $\mathscr{B}^{\rm s}_k$ is caused by the definition of $\ubar{z}^{(k)}$ in  \eqref{020}. Let us remark that although $\widehat{B}^{\rm s}_k$ and $\widehat{B}^{\rm u}_k$ have at least one common point since they are linked, while in general, $\mathscr{B}^{\rm s}_k$ and $\mathscr{B}^{\rm u}_k$ may be disjoint. See Figure \ref{fig5} for a conceptual picture of these bridges defined on the tangency curve $L$. Finally, let $\mathscr{A}^{{\rm s,u}}_k$, $\widehat{A}^{{\rm s,u}}_k$ and $A^{{\rm s,u}}_k$ be the pre-images of $\mathscr{B}^{{\rm s,u}}_k$, $\widehat{B}^{{\rm s,u}}_k$ and $B^{{\rm s,u}}_k$ under $f^2$ respectively, located on $\widetilde{L}$. 
\begin{figure}[hbt]
\centering
\scalebox{0.9}{
\includegraphics[clip]{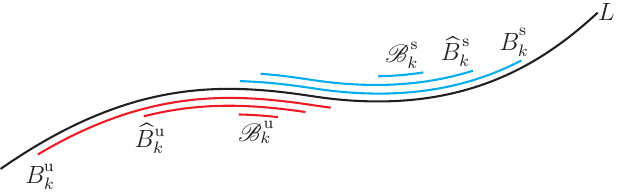}
}
\caption{$\rm s$-bridges and $\rm u$-bridges on $L$.
} 
\label{fig5}
\end{figure}

We have the following two claims.
\begin{clm}\label{141}
The ${\rm s}$-bridges
$A_k^{\rm s}\ (k\in\mathbb{N})$ are uniformly pairwise disjoint.  
\end{clm}
\begin{proof}
First, let us note that $A^{\rm s}_k=f^{-2}(B^{\rm s}_k)$ for every $k$. Since $B^{\rm s}_k$ are obtained by Lemma \ref{lem2}, they are pairwise disjoint (see Step 1 in its proof). Thus, we conclude that $A_k^{\rm s}\ (k\in\mathbb{N})$ are also pairwise disjoint. It remains to show the uniformity of the disjointness. 

For every $k\in\mathbb{N}$, as $B^{\rm s}_k\subset L$ is an ${\rm s}$-bridge of generation $s_k$, say, $B^{\rm s}_k=B^{\rm s}_k(s_k; \underline{w}^{(k)})$ with $\underline{w}^{(k)}\in\{0,1\}^{s_k}$. By \eqref{113} and \eqref{114}, we see that $A^{\rm s}_k=A^{\rm s}_k(s_k; \underline{w}^{(k)})\subset \widetilde{L}$ can also be seen as the pre-image of $Br^{\rm s}_k=Br^{\rm s}_k(s_k; \underline{w}^{(k)})\subset I^{\rm s}$ under $\pi_{\mathcal{F}^{\rm u}}|_{\widetilde{L}}$, where $\pi_{\mathcal{F}^{\rm u}}|_{\widetilde{L}}$ is the restriction of $\pi_{\mathcal{F}^{\rm u}}$ to $\widetilde{L}$, see \eqref{116}.

Since $\pi_{\mathcal{F}^{\rm u}}|_{\widetilde{L}}$ is almost affine, to prove the claim, it suffices to show that $Br^{\rm s}_k\ (k\in\mathbb{N})$ are uniformly pairwise disjoint. Indeed, note that each $Br^{\rm s}_k$ is a bridge of the Cantor set $\Lambda^{\rm s}_f$ defined in \eqref{111}, we only need to show that any gap of $\Lambda^{\rm s}_f$ occupies a relatively large proportion in length compared to the length of its adjacent bridges. To see this, let us take an arbitrary {\rm s}-gap of $\Lambda^{\rm s}_f$, say $Ga^{\rm s}$ (recall the related definitions in Subsection \ref{071}). Suppose $Br^{\rm s}$ is either of its two adjacent bridges. By Definition \ref{defthick} and the choice of $\theta$ in \eqref{117}, and notice that $f$ is contained in $\mathcal{U}_0\subset \mathcal{U}^r_F$, we conclude that 
\[\dfrac{|Br^{\rm s}|}{|Ga^{\rm s}|}<\theta,\]
which immediately yields 
\[|Ga^{\rm s}|>\theta^{-1}|Br^{\rm s}|.\]
Since $\theta$ (hence $\theta^{-1}$) is a positive constant independent of $f$, we complete the proof of the claim. 
\end{proof}





\begin{clm}\label{145}
For every $k\in\mathbb{N}$ and $\ubar{z}^{(k)}$, $n_k$ defined in \eqref{020}-\eqref{108}, each leaf of $\mathcal{F}^{\rm u}$ inside $\mathbb{B}r^{\rm s}(n_k;[\ubar{z}^{(k)}]^{-1})$ intersects $\widetilde{L}$ transversally.
\end{clm}

\begin{proof}
Let us fix an arbitrary $k\in\mathbb{N}$. On the one hand, 
it follows from the definition of $\widehat{B}^{\rm s}_{k+1}$ in \eqref{070} that 
\[\widehat{B}^{\rm s}_{k+1}=\widehat{B}^{\rm s}_{k+1}(\widehat{s}_{k+1};\widehat{\ubar{w}}^{(k+1)})\subset B^{\rm s}_{k+1}\subset L.\] 
Combining this fact with the definition of $\widehat{A}^{\rm s}_{k+1}$ before Claim \ref{141}, we have 
\[\widehat{A}^{\rm s}_{k+1}=\widehat{A}^{\rm s}_{k+1}(\widehat{s}_{k+1};\widehat{\ubar{w}}^{(k+1)})\subset \widetilde{L},\] 
since $\widetilde{L}$ is the $f^{-2}$-image of $L$ (see Subsection \ref{110}). Therefore, by \eqref{135} and \eqref{114}, each leaf of $\mathcal{F}^{\rm u}$ inside $\mathbb{B}r^{\rm s}_{k+1}(\widehat{s}_{k+1};\widehat{\ubar{w}}^{(k+1)})$ intersects $\widetilde{L}$ transversally. On the other hand, as \eqref{020} gives 
\[
[\ubar{z}^{(k)}]^{-1}=\big[\widehat{\ubar{z}}^{(k)}\widehat{\ubar{v}}^{(k)}[\widehat{\ubar{w}}^{(k+1)}]^{-1}\big]^{-1}=\widehat{\ubar{w}}^{(k+1)}[\widehat{\ubar{v}}^{(k)}]^{-1}[\widehat{\ubar{z}}^{(k)}]^{-1},
\]
we obtain that $\mathbb{B}r^{\rm s}(n_k;[\ubar{z}^{(k)}]^{-1})$ is a sub-bridge stripe of $\mathbb{B}r^{\rm s}_{k+1}(\widehat{s}_{k+1};\widehat{\ubar{w}}^{(k+1)})$. Considering the above two aspects together, the conclusion follows immediately.
\end{proof}
 
\subsection{Creation of the critical chain}\label{ccl}
To give the next lemma, we need some notational preparations. Let $\mathcal{U}_0$ be the open set given by Claim \ref{140} and $f$ an arbitrary element of $\mathcal{U}_0$. With the notations defined in the previous subsection, for every $k\in\mathbb{N}$, let
$\widetilde{L}_k:=f^{n_k}(L\cap \mathbb{B}r^{\rm u}(n_k;\underline{z}^{(k)}))$. As a result of Claim \ref{145}, we can assume that $\widetilde{L}_k$ intersects $\widetilde{L}$ transversely at 
\begin{equation}\label{074}
\bm{q}_k\in \mathscr{A}^{\rm s}_{k+1}\subset \widetilde{L}.
\end{equation} We also define 
\begin{equation}\label{028}
\begin{split}
\bm{x}_k&:=f^{-n_k}(\bm{q}_k)\in \mathscr{B}^{\rm u}_k\subset L,\quad
\bm{y}_k:=f^{2}(\bm{q}_k)\in \mathscr{B}^{\rm s}_{k+1}\subset L, \\
\bm{r}_k&:=f^{-2}(\bm{x}_k)\in \mathscr{A}^{\rm u}_k\subset \widetilde{L}.
\end{split}
\end{equation}
In other words, the following transfer sequence
\[\bm{r}_k \xrightarrow{f^2} \bm{x}_k \xrightarrow{f^{n_k}} \bm{q}_k  \xrightarrow{f^2} \bm{y}_k \]
is well defined for every $k\in\mathbb{N}$. See Figure \ref{fig6}. We need to point out that all these points $(\bm{r}_k, \bm{x}_k, \bm{q}_k, \bm{y}_k)$ certainly depend on $f$.
\begin{figure}[hbt]
\centering
\scalebox{0.94}{
\includegraphics[clip]{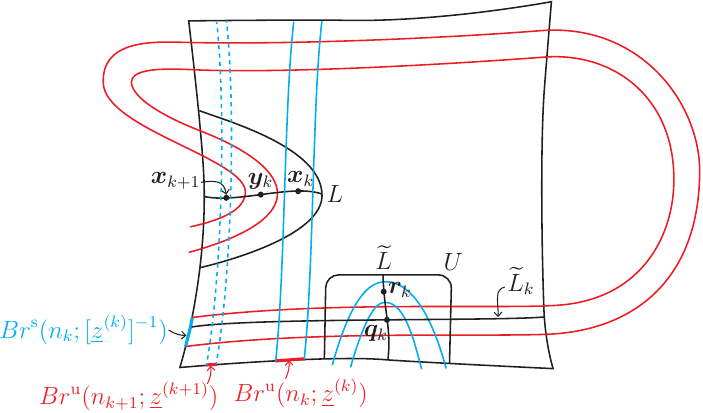}
}
\caption{Locations of $\bm{r}_k$, $\bm{x}_k$, $\bm{q}_k$ and $\bm{y}_k$.} 
\label{fig6}
\end{figure}

The main result of this section is the following so-called Critical Chain Lemma. Let us explain a little more. For every $k\in\mathbb{N}$, we have \[f^{n_k+2}(\bm{x}_k)=\bm{y}_k\] according to the above transfer sequence. Now, if $\bm{y}_k$ happens to be $\bm{x}_{k+1}$ exactly, then we are allowed to act $f^{n_{k+1}+2}$ once again on it, obtaining 
\[f^{n_{k+1}+2}\circ f^{n_{k}+2}(\bm{x}_k)=f^{n_{k+1}+2}(\bm{x}_{k+1}) =\bm{y}_{k+1}.\] 
Moreover, if $\bm{y}_k=\bm{x}_{k+1}$ holds for every $k$, we thus have the following infinite transfer sequence of $\bm{x}_k$, 
called a \emph {critical chain}:
\[\bm{x}_1 \xrightarrow{f^{n_1+2}} 
\bm{x}_{2} \xrightarrow{f^{n_{2}+2}} 
\cdots 
\xrightarrow{f^{n_{k-2}+2}}
\bm{x}_{k-1}  
\xrightarrow{f^{n_{k-1}+2}}
 \bm{x}_{k} 
 \xrightarrow{f^{n_{k}+2}} 
  \cdots\]
which will be very useful when we construct the non-trivial wandering domain. However, it is quite difficult to meet such coincidental conditions for $\bm{y}_k$ and $\bm{x}_{k+1}$ in general. Now, we are in the position to state the following lemma which yields the desired condition.
\begin{lem}[Critical Chain Lemma]\label{lem5} 
For every $\varepsilon>0$, $\widehat{m}_k\in \mathbb{N}$ and $\widehat{\ubar{v}}^{(k)}\in \{0,1\}^{\widehat{m}_k}$, there exists an $\varepsilon$-small $C^r$ perturbation $g$ of $f$ such that, 
for $n_k$, $\ubar{z}^{(k)}$, $\mathscr{B}^{\rm s}_k$, $\mathscr{B}^{\rm u}_k$, $\bm{x}_k$, $\bm{y}_k$ ($k=1,2,\ldots$) defined in \eqref{108}, \eqref{020}, \eqref{021}, \eqref{028}, the followings hold:
\begin{itemize}
\item[(1)] $\bm{y}_{k}=\bm{x}_{k+1}$, hence $g^{n_k+2}(\bm{x}_{k})=\bm{x}_{k+1}$,
\item[(2)] $Dg^{n_k+2}(T_{\bm{x}_{k}}\mathcal{F}^{\rm u})=T_{\bm{x}_{k+1}}\mathcal{F}^{\rm s}$,
\item[(3)] $\widehat{u}_k+\widehat{s}_{k+1}\le Ck$ where $C$ is a constant independent of $k$.
\end{itemize}
\end{lem}
\begin{proof}
For every $k\in\mathbb{N}$, we have constructed $\mathscr{B}^{\rm s}_k$ and $\mathscr{B}^{\rm u}_k$ in \eqref{021}. 
For every $k\ge 2$, recall that $\bm{y}_{k-1}\in \mathscr{B}^{\rm s}_k\subset\widehat{B}^{\rm s}_k$ and $\bm{x}_{k}\in \mathscr{B}^{\rm u}_k\subset\widehat{B}^{\rm u}_k$. Since $(\widehat{B}^{\rm s}_k,\widehat{B}^{\rm u}_k)$ was selected as a linked pair in \eqref{070}, if we denote by $|\bm{y}_{k-1}\bm{x}_k|_L$ the arc-length of the segment on $L$ which connects $\bm{y}_{k-1}$ and $\bm{x}_k$, then we have
\begin{align*}
|\bm{y}_{k-1}\bm{x}_k|_L
&\le |\widehat{B}^{\rm s}_k|+|\widehat{B}^{\rm u}_k|
\le |\widehat{B}^{\rm s}_k|+\overline{\sigma}|\widehat{B}^{\rm s}_k|\\
&\le (1+\overline{\sigma})\ubar{\lambda}^{kN+1}\le 2\ubar{\lambda}^{kN},
\end{align*}
where the second and  third inequalities follow from 
\eqref{022}, and the last inequality holds because $\ubar{\lambda}(1+\bar{\sigma})<2$ by \eqref{046}. Notice that $$\bm{q}_{k-1}=f^{-2}(\bm{y}_{k-1})\in \mathscr{A}^{\rm s}_{k}\subset \widetilde{L}\quad \text{and}\quad \bm{r}_k=f^{-2}(\bm{x}_k)\in \mathscr{A}^{\rm u}_k\subset \widetilde{L}.$$
If we denote by $\bm{\zeta}_k=\bm{\zeta}_k(N)$ the vector which starts at $\bm{q}_{k-1}$ and ends at $\bm{r}_k$, it follows that
\begin{equation}\label{023}
\|\bm{\zeta}_k\|\le |\bm{q}_{k-1}\bm{r}_k|_{\widetilde{L}}\le C_1\ubar{\lambda}^{kN}
\end{equation}
for some constant $C_1$ which only depends on the neighborhood $\mathcal{U}_F^r$. 
\begin{figure}[hbt]
\centering
\scalebox{0.9}{
\includegraphics[clip]{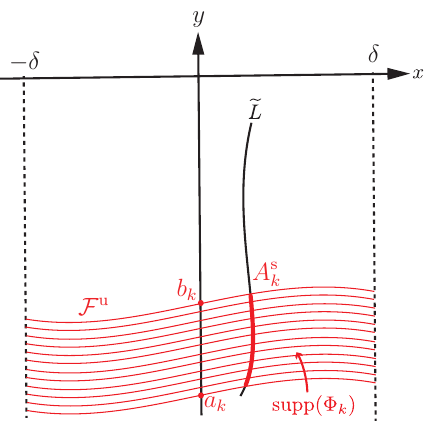}
}
\caption{The image of $\mbox{supp}(\Phi_k)$.} 
\label{fig7}
\end{figure}

Denote by $[a_k,b_k]$ the projection of $A^{\rm s}_k$ to $\{x=0\}$ (i.e. the $y$-axis) along $\mathcal{F}^{\rm u}$. 
Since $A^s_k\ (k=1,2,\ldots)$ are uniformly pairwise disjoint by Claim \ref{141}, we see that $[a_k,b_k]$ are also $\rho$-uniformly pairwise disjoint for some $\rho>0$. We recall that $U$ is the small neighborhood of $(0,-a_{\rm s})$ in $(-2,2)^2$ given in Subsection \ref{112}. Let us assume that the $\pi_x$-image of all points in $U$ is contained in $[-\delta,\delta]$. Let 
\[\chi(x):=\phi_{\frac{1}{4},[-\delta,\delta]}(x)\quad \text{and}\quad \chi_k(y):=\phi_{\frac{1}{10}\rho,[a_k,b_k]}(y)
\]  
be functions defined on the $x$-axis and the $y$-axis respectively.
It follows that $\chi_k$ $(k=2,3,\dots)$ have pairwise disjoint supports. Define
\[\Phi_k(\bm{x}):=\chi(\pi_{x}(\bm{x}))\cdot\chi_k(\pi^0_{\mathcal{F}^{\rm u}}(\bm{x})).\]
Here, $\pi^0_{\mathcal{F}^{\rm u}}$ is the projection to $\{x=0\}$ along leaves of $\mathcal{F}^{\rm u}$. According to the notations and properties of the bump function listed at the beginning of this section, one easily deduces that 
\[\|\Phi_k\|_{C^r}\lesssim \dfrac{1}{|A^{\rm s}_k|^r}\] for every $k=2,3,\dots$. See Figure \ref{fig7} for the image of ${\rm supp}(\Phi_k)$.

Recall that $L$ is the tangency curve between $f^2(\mathcal{F}^{\rm u})$ and $\mathcal{F}^{\rm s}$, therefore, leaves of $\mathcal{F}^{\rm u}$ and leaves of $f^{-2}(\mathcal{F}^{\rm s})$ tangent to each other along $\widetilde{L}=f^{-2}(L)$ to which $\bm{q}_{k-1}$ and $\bm{r}_k$ belong. If we denote by $T_{\bm x}\mathcal{F}^{{\rm s}({\rm u})}$ 
the tangent line of the leaf of $\mathcal{F}^{{\rm s}({\rm u})}$ passing through ${\bm x}$, then we have $T_{\bm{r}_k}(f^{-2}(\mathcal{F}^{\rm s}))=T_{\bm{r}_k}\mathcal{F}^{\rm u}$. Moreover, it follows from \cite[Appendix 1, Theorem 8]{PT93} that $T_{\bm x}\mathcal{F}^{\rm u}$ (indeed, $T_{\bm x}\mathcal{F}^{\rm s}$ also) $C^1$-depends on $\bm{x}$. Combining these facts with \eqref{023} and using the mean value theorem, we see that the angle
\begin{equation*}
\omega_k:=\angle \big( T_{\bm{q}_{k-1}}\mathcal{F}^{\rm u} , T_{\bm{r}_k}(f^{-2}(\mathcal{F}^{\rm s}))\big)
=\angle \big( T_{\bm{q}_{k-1}}\mathcal{F}^{\rm u}, T_{\bm{r}_k}\mathcal{F}^{\rm u}\big)
\end{equation*} satisfies
\begin{equation}\label{025}
|\omega_k|\lesssim \ubar{\lambda}^{kN}.
\end{equation} 
See Figure \ref{fig8}.
\begin{figure}[hbt]
\centering
\scalebox{0.9}{
\includegraphics[clip]{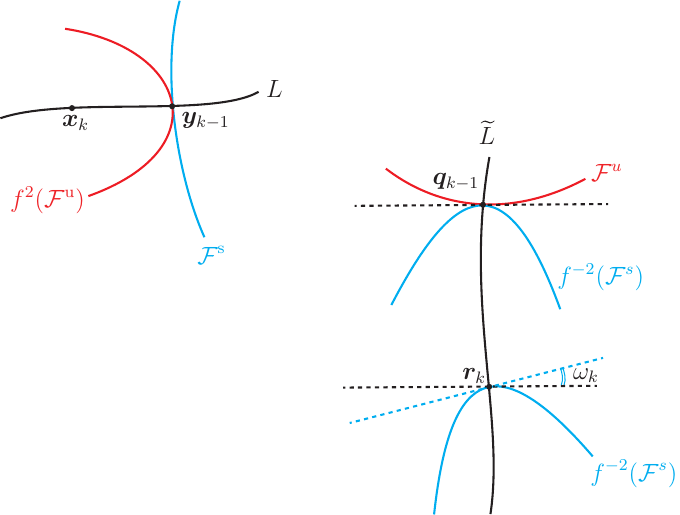}
}
\caption{The angle $\omega_k$.} 
\label{fig8}
\end{figure}

Let $D_k:\mathbb{R}^2\to \mathbb{R}^2$ be the rotation transformation with angle $\omega_k$ at $\bm{q}_{k-1}$, that is, 
\[D_k(\bm{x}):=\bm{q}_{k-1}+\begin{pmatrix}
\cos \omega_k & -\sin \omega_k\\
\sin \omega_k &  \cos \omega_k
\end{pmatrix}(\bm{x}-\bm{q}_{k-1}).\]
Then, if $E:\mathbb{R}^2\to \mathbb{R}^2$ is the identity transformation, we have by \eqref{025} that
\begin{equation}\label{026}
\|D_k-E\|_{C^r}\sim |\omega_k|\lesssim\ubar{\lambda}^{kN}.
\end{equation}
Consider the linear transformation $\xi_k$ on $\mathbb{R}^2$ defined by
\begin{align*}
\xi_k(\bm{x}):&=\bm{\zeta}_k+D_k(\bm{x})= \bm{r}_k-\bm{q}_{k-1}+D_k(\bm{x})\\
&=\bm{r}_k+\begin{pmatrix}
\cos \omega_k & -\sin \omega_k\\
\sin \omega_k &  \cos \omega_k
\end{pmatrix}(\bm{x}-\bm{q}_{k-1}).
\end{align*}
This definition immediately gives
\begin{itemize}
\item $\xi_k(\bm{q}_{k-1})=\bm{r}_k$ and
\item $\xi_k(\bm{q}_{k-1}) T_{\bm{q}_{k-1}}\mathcal{F}^{\rm u}=T_{\bm{r}_k}(f^{-2}(\mathcal{F}^{\rm s}))$.
\end{itemize}
Moreover, we have
\begin{equation}\label{035}
\begin{split}
\|(\xi_k-\mbox{id})|_{U}\|_{C^r}
&\le \max_{\bm{x}\in U}\|(D_k(\bm{x})-\bm{x})+(\bm{r}_k-\bm{q}_{k-1})\|_{C^r}\\
&\le\max_{\bm{x}\in U} \|D_k(\bm{x})-\bm{x}\|_{C^r}+\max_{\bm{x}\in U} \|\bm{r}_k-\bm{q}_{k-1}\|_{C^r}\\
&\le \left\|\begin{pmatrix}
\cos \omega_k & -\sin \omega_k\\
\sin \omega_k &  \cos \omega_k
\end{pmatrix}-E\right\|_{C^r}+\|\bm{r}_k-\bm{q}_{k-1}\|_{C^r}\lesssim \ubar{\lambda}^{kN},
\end{split}
\end{equation}
where the last inequality comes from \eqref{023} and \eqref{026}. Here, we recall that $U$ is the small neighborhood of $(0,-a_{\rm s})$ in $(-2,2)^2$ given in  Subsection \ref{112}.

\begin{clm}\label{clm4}
There exists constant $C_N>0$ satisfying $C_N\to 0$ as $N\to\infty$, such that
\[\sum_{k=2}^\infty \dfrac{\ubar{\lambda}^{kN}}{|A^{\rm s}_k|^r}\le C_N\]for every sufficiently large $N\in\mathbb{N}$.
\end{clm}

\begin{proof}[Proof of Claim]
Indeed, recall that the sequence $B^{\rm s}_k$ is obtained from Lemma \ref{lem2} with generations $s_k$ satisfying $s_k\le s_0+kN_{\rm s}$ for some $s_0$. Thus, for every $k=2,3,\dots$, we have
\[|A^{\rm s}_k|\sim |B^{\rm s}_k|\gtrsim \ubar{\lambda}^{s_k}\gtrsim \ubar{\lambda}^{kN_{\rm s}},\]
which implies that, if $N>rN_{\rm s}$, then there is some constant $C'>0$
 independent of $k$ such that
\begin{align*}
\sum_{k=2}^\infty \dfrac{\ubar{\lambda}^{kN}}{|A^{\rm s}_k|^r}
&\le C' \sum_{k=2}^\infty \dfrac{\ubar{\lambda}^{kN}}{\ubar{\lambda}^{krN_s}}
= C' \sum_{k=2}^\infty \ubar{\lambda}^{k(N-rN_{\rm s})}\\
&=\dfrac{C' \ubar{\lambda}^{2(N-rN_{\rm s})}}{1-\ubar{\lambda}^{N-rN_{\rm s}}}=:C_N.
\end{align*}
Thus we finish the proof of the Claim \ref{clm4} by noticing $C_N\to 0$ as $N\to\infty$.
\end{proof}

Let us continue the proof of the lemma. Let $$\bm{\zeta}=\bm{\zeta}(N):=(\bm{\zeta}_2,\bm{\zeta}_3,\dots ,\bm{\zeta}_k,\ldots)$$ be an infinite sequence of vectors, which is called a {\it perturbation vector sequence}. We claim that the $\bm{\zeta}$-related map sequence
\[\Phi_{\bm{\zeta},l}(\bm{x}):=\bm{x}+\sum_{k=2}^l \Phi_k(\bm{x})(\xi_k(\bm{x})-\bm{x})\]
forms a Cauchy sequence. Indeed, suppose $m$ and $n$ are any pair of positive integers with $m>n$, we thus have, by \eqref{035}, that  
\begin{align*}
\big\|\Phi_{\bm{\zeta},m}-\Phi_{\bm{\zeta},n}\big\|_{C^r}
&=\left\|\sum_{k=n+1}^m \Phi_k (\xi_k-\mbox{id})\right\|_{C^r}
\le\sum_{k=n+1}^m \big\| \Phi_k (\xi_k-\mbox{id})\big\|_{C^r}\\
&\lesssim \sum_{k=n+1}^m \big\|\Phi_k\big\|_{C^r}  \big\|(\xi_k-\mbox{id})|_{U}\big\|_{C^r}
\lesssim  \sum_{k=n+1}^m\dfrac{\ubar{\lambda}^{kN}}{|A^{\rm s}_k|^r}.
\end{align*}
Since the series in Claim \ref{clm4} converges, for any $\varepsilon_0>0$, there is a sufficiently large $N_{0}\in \mathbb{N}$ such that 
if $m>n>N_0$, we have
\[\big\|\Phi_{\bm{\zeta},m}-\Phi_{\bm{\zeta},n}\big\|_{C^r}<\varepsilon_0.\]
As a result, we are allowed to define 
\[\Phi_{\bm{\zeta}}(\bm{x}):=\lim\limits_{l\to\infty}\Phi_{\bm{\zeta},l}(\bm{x})=\bm{x}+\sum_{k=2}^\infty \Phi_k(\bm{x})(\xi_k(\bm{x})-\bm{x}).\]  
By definition, it is not hard to verify that $\Phi_{\bm{\zeta}}$ satisfies 
\begin{itemize}
\item $\Phi_{\bm{\zeta}}(\bm{q}_{k-1})=\bm{r}_k$ and
\item $\Phi_{\bm{\zeta}}(\bm{q}_{k-1}) T_{\bm{q}_{k-1}}\mathcal{F}^{\rm u}=T_{\bm{r}_k}(f^{-2}\mathcal{F}^{\rm s})$
\end{itemize}
for every $k=2,3,\dots$.

\vspace{3mm}
Now, let us finish the proof of Lemma \ref{lem5}. 
Indeed, we notice that 
\begin{align*}
\mbox{dist}_{C^r}(\Phi_{\bm{\zeta}}, \mbox{id})
&=\max_{\bm{x}\in M} \left\|\sum_{k=2}^\infty \Phi_k(\bm{x})(\xi_k(\bm{x})-\bm{x})\right\|_{C^r}\\
&\le \sum_{k=2}^\infty \big\|\Phi_k\big\|_{C^r} \big\|(\xi_k-\mbox{id})|_{U}\big\|_{C^r}
\lesssim\sum_{k=2}^\infty  \dfrac{\ubar{\lambda}^{kN}}{|A^{\rm s}_k|^r}.
\end{align*}
Therefore, given an arbitrarily small $\varepsilon>0$ as in the hypothesis of Lemma \ref{lem5}, according to Claim \ref{clm4} with 
a sufficiently large $N$, it holds that 
$\mbox{dist}_{C^r}(\Phi_{\bm{\zeta}}, \mbox{id})< \varepsilon\|f\|_{C^r}^{-1}$. 
Define 
\[g:=f\circ \Phi_{\bm{\zeta}}:M\to M.\]
 Then, we have
\[\mbox{dist}_{C^r}(g,f)\le\mbox{dist}_{C^r}(\Phi_{\bm{\zeta}},\mathrm{id})\|f\|_{C^r}
<\varepsilon.\]
In other words, we see that $g$ is an $\varepsilon$-small $C^r$-perturbation of $f$. 
Since $\Diff^r(M)$ is open in the space of $C^r$ self-maps of $M$, we conclude that $g$ is also an element of $\Diff^r(M)$. 

It remains to verify that $g$ satisfies the conclusion of Lemma \ref{lem5}. 
For (1), we have
\begin{align*}
g^{n_{k-1}+2}(\bm{x}_{k-1})
&=(f\circ \Phi_{\bm{\zeta}})^2\circ (f\circ \Phi_{\bm{\zeta}})^{n_{k-1}}(\bm{x}_{k-1})\\
&=(f\circ \Phi_{\bm{\zeta}})^2\circ f^{n_{k-1}}(\bm{x}_{k-1})
=(f\circ \Phi_{\bm{\zeta}})^2(\bm{q}_{k-1})\\
&=f\circ \Phi_{\bm{\zeta}}\circ f(\bm{r}_{k})=f^2(\bm{r}_k)=\bm{x}_k.
\end{align*}
For (2), we have
\begin{align*}
Dg^{n_{k-1}+2}(\bm{x}_{k-1})T_{\bm{x}_{k-1}}\mathcal{F}^{\rm u}
&=Dg^2(\bm{q}_{k-1})T_{\bm{q}_{k-1}}\mathcal{F}^{\rm u}\\
&=D(f\circ \Phi_{\bm \zeta}\circ f)(\bm{r}_k)T_{\bm{r}_k}(f^{-2}(\mathcal{F}^{\rm s}))=T_{\bm{x}_k}\mathcal{F}^{\rm s}. 
\end{align*}
For (3), note that we have
\[\ubar{\lambda}^{(k+1)N}\lesssim |\widehat{B}^{\rm s}_{k+1}(\widehat{s}_{k+1};\widehat{\ubar{w}}^{(k+1)})|\lesssim \overline{\lambda}^{\widehat{s}_{k+1}},\]
where the first inequality comes from \eqref{022} and the second inequality comes from \eqref{073}. 
Thus, one can suppose that
\[\widehat{s}_{k+1}\le kN\dfrac{\log \ubar{\lambda}}{\log \overline{\lambda}},\]
if necessary replacing $N$ by a larger integer.
Similarly, we can also deduce that 
\[\widehat{u}_{k}\le kN\dfrac{\log \ubar{\lambda}}{\log \ubar{\sigma}^{-1}}.\]
Therefore, if we take 
\[C>\max\left\{\dfrac{N\log \ubar{\lambda}}{\log \overline{\lambda}},\  \dfrac{N\log \ubar{\lambda}}{\log \ubar{\sigma}^{-1}}\right\}\]
large enough, then $\widehat{s}_{k+1}+\widehat{u}_k<Ck$ holds for every $k$.
We now complete the proof of Lemma \ref{lem5}. 
\end{proof}

\section{Rectangle Lemma}\label{065}
In this section, we will construct a sequence of rectangles. Each rectangle in this sequence is located around ${\bm x}_k$ of the critical chain obtained in the previous section. It will be a non-trivial wandering domain that we aim to construct. To this end, let us begin with some preliminary work.

Choose an arbitrary element $f$ of $\mathcal{U}_0$. With the notations defined in the previous section, let us take 
\begin{equation}\label{109}
\widehat{m}_k=k^2
\end{equation} in Lemma \ref{lem5}. Since $\widehat{u}_k+\widehat{s}_{k+1}<Ck$ for some constant $C>0$ independent of $k$ as indicated in Lemma \ref{lem5} (3), it is not hard to see that $n_k$ defined in \eqref{108} is increasing and 
\[\dfrac{\widehat{u}_k+\widehat{s}_{k+1}}{\widehat{m}_k}=\dfrac{O(k)}{k^2}\to 0\] as $k\to\infty$. Moreover, since 
\[\dfrac{n_{k+1}}{n_k}=
\dfrac{(k+1)^2+O(k+1)}{k^2+O(k)}\to 1\] as $k\to \infty$, for every $\eta>0$, it holds that 
\begin{equation}\label{161}
n_{k+1}<(1+\eta)n_k
\end{equation}
for every sufficiently large $k$. Let us assume that this inequality holds for every $k\in\mathbb{N}$ for notational simplicity (otherwise it is enough to translate the subscript). In addition, we can require that $\eta>0$ is so small that 
\begin{equation}\label{079}
\overline{\lambda}\overline{\sigma}^{\frac{1+2\eta}{1-\eta}}<1
\end{equation} 
holds by $\overline{\lambda}\overline{\sigma}<1$, see \eqref{lambdasigma}. 
\begin{lem}[Rectangle Lemma]\label{lem6}
For every $f\in \mathcal{U}_0$, there exist an arbitrarily small $C^r$ perturbation $g$ of $f$ and a sequence of (topological) rectangles $R_k\ (k=1,2,\dots)$ such that each $R_k$ has $\bm{x}_k$ as its center and satisfies the following properties:
\begin{itemize}
\item[(1)] $\mathrm{diam}(R_k)\to 0$ as $k\to\infty$,
\item[(2)] 
for the rectangle $Q=[-1,1]^2$, 
\[R_k\subset \mathbb{G}a_k^{\rm u}(n_k;\ubar{z}^{(k)})\cap \left(Q \backslash (g(S_{0,g})\cup g(S_{1,g}))\right),\]
 in particular, $\{R_k\}$ are pairwise disjoint,
\item[(3)] $g^{n_k+2}(R_k)\subset R_{k+1}$.
\end{itemize}
\end{lem}

\begin{proof}
Let $\widehat{m}_k$ be selected as in \eqref{109}. Fix an arbitrarily small $\varepsilon>0$. By applying Lemma \ref{lem2} and Lemma \ref{lem5} to $f$ sequentially, we obtain $g$ which satisfies, in particular, items (1) and (2) of Lemma \ref{lem5}. By shrinking $\varepsilon$ in advance if necessary, we can require the $C^r$-distance between $g$ and $f$ to be as small as we want. As a result, for every $(x,y)\in U$, we can write
\begin{equation}
g^2(x,y)=(-\tilde{a}_{\rm u}+\tilde{\mu}-\tilde{\beta}x^2-\tilde{\gamma}(y+\tilde{a}_{\rm s}),-\tilde{\alpha}x)+h(x,y),
\end{equation}
where all of the coefficients $\tilde{\alpha}$, $\tilde{\beta}$, $\tilde{\gamma}$, $\tilde{\mu}$, $\tilde{a}_{\rm u}$, $\tilde{a}_{\rm s}$ are  $\varepsilon$-close to $\alpha$, $\beta$, $\gamma$, $\mu$, $a_{\rm u}$, $a_{\rm s}$ respectively, and $h(x,y)$ is the higher order terms containing $o(x^2)$ and $o(y)$.

Let $\rho\in(0,1)$ be a constant which will be fixed later, and define
\begin{equation}\label{083}
b_k:=\rho\tilde{\beta}^{-1}\overline{\sigma}^{-\sum_{i=0}^\infty\frac{n_{k+i}}{2^i}}.
\end{equation} 
It follows immediately from (see \eqref{161})
\[2n_k=\sum_{i=0}^\infty\dfrac{n_k}{2^i}\le\sum_{i=0}^\infty\dfrac{n_{k+i}}{2^i}\le n_k\sum_{i=0}^\infty\left(\dfrac{1+\eta}{2}\right)^i=\dfrac{2n_k}{1-\eta} \] that
\begin{equation}\label{029}
\rho\tilde{\beta}^{-1}\overline{\sigma}^{-\frac{2}{1-\eta}n_k}\le b_k \le \rho\tilde{\beta}^{-1}\overline{\sigma}^{-2n_k}.
\end{equation}
The desired rectangle $R_k$ will be defined by taking leaves of $\mathcal{F}^{\rm s}$ and $\mathcal{F}^{\rm u}$ as its boundary. Let us be more precise. Denote by $\mathcal{F}^{{\rm s}({\rm u})}(\bm x)$ the leaf of $\mathcal{F}^{{\rm s}({\rm u})}$ passing through ${\bm x}$. Take $\bm{x}_k^l, \bm{x}_k^r\in \mathcal{F}^{\rm u}(\bm{x}_k)$ and $\bm{x}_k^t, \bm{x}_k^b\in \mathcal{F}^{\rm s}(\bm{x}_k)$ with 
\begin{align}
|\bm{x}_k^l \bm{x}_k|_{\mathcal{F}^{\rm u}(\bm{x}_k)}&=|\bm{x}_k \bm{x}_k^r|_{\mathcal{F}^{\rm u}(\bm{x}_k)}=b_k/2,\label{030}\\
|\bm{x}_k^t \bm{x}_k|_{\mathcal{F}^{\rm s}(\bm{x}_k)}&=|\bm{x}_k \bm{x}_k^b|_{\mathcal{F}^{\rm s}(\bm{x}_k)}=10\tilde{\alpha}\tilde{\beta}^{-\frac{1}{2}}\sqrt{b_k}.\label{031}
\end{align}
Thus, the four leaves $\mathcal{F}^{\rm u}(\bm{x}_k^t)$, $\mathcal{F}^{\rm u}(\bm{x}_k^b)$, $\mathcal{F}^{\rm s}(\bm{x}_k^l)$ and $\mathcal{F}^{\rm s}(\bm{x}_k^r)$ bound a rectangle $R_k$ whose top, bottom, left and right boundaries are sub-arcs of these leaves centered at 
$\bm{x}_k^t$, $\bm{x}_k^b$, $\bm{x}_k^l$ and $\bm{x}_k^r$, respectively. 
Briefly, we call $b_k$ and $20\tilde{\alpha}\tilde{\beta}^{-\frac{1}{2}}\sqrt{b_k}$ the {\it width} and {\it height} of $R_k$. See Figure \ref{fig9}.
\begin{figure}[hbt]
\centering
\scalebox{0.9}{
\includegraphics[clip]{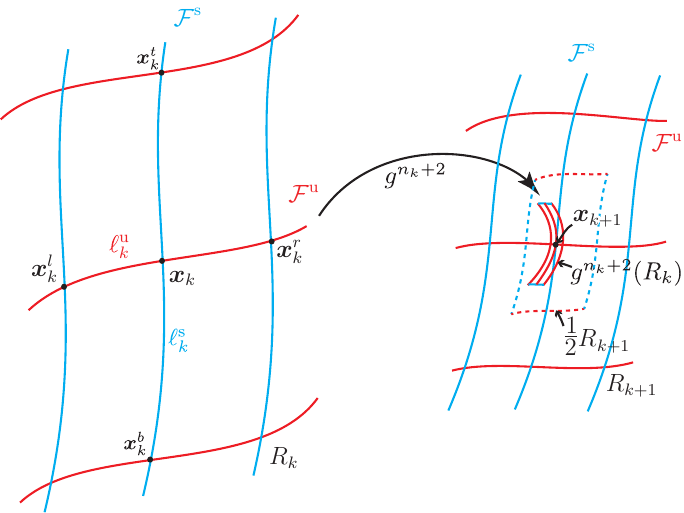}
}
\caption{The rectangles $R_k$ and $R_{k+1}$.} 
\label{fig9}
\end{figure}

Now, it remains to verify that $R_k$ satisfies the conclusions (1)-(3).

For (1), notice that
\[\mbox{diam}(R_k)\lesssim \max\big\{|\bm{x}_k^l \bm{x}_k|_{\mathcal{F}^{\rm u}(\bm{x}_k)},|\bm{x}_k \bm{x}_k^r|_{\mathcal{F}^{\rm u}(\bm{x}_k)},|\bm{x}_k^t \bm{x}_k|_{\mathcal{F}^{\rm s}(\bm{x}_k)},  |\bm{x}_k \bm{x}_k^b|_{\mathcal{F}^{\rm s}(\bm{x}_k)}\big\}. \] Hence, (1) is an immediate consequence of \eqref{030} and \eqref{031} together with the fact that $b_k\to 0$ as $k\to\infty$.

For (2), first, let us note that by \eqref{074} and \eqref{028}, we have
\begin{equation}\label{075}
\bm{x}_k\in G^{\rm u}_k(n_k;\ubar{z}^{(k)})
\end{equation}
for every $k$. Here, we recall that $G^{\rm u}_k(n_k;\ubar{z}^{(k)})$ is the $\rm u$-gap of generation $n_k$ and itinerary $\ubar{z}^{(k)}$ on the tangency curve $L_g$, see the end of Subsection \ref{071}. 
Since gap strips with different itineraries are pairwise disjoint, 
to show the pairwise disjointness of $R_k$, it is sufficient to prove that each $R_k$ is completely contained in the middle component of 
\[\mathbb{G}a_k^{\rm u}(n_k;\ubar{z}^{(k)})\cap (Q \backslash (g(S_{0,g})\cup g(S_{1,g}))),\] 
where $\mathbb{G}a_k^{\rm u}(n_k;\ubar{z}^{(k)})$ is the gap strip associated to $Ga_k^{\rm u}(n_k;\ubar{z}^{(k)})$ and $S_{0,g}, S_{1,g}$ are the continuations of $S_0, S_1$ for $g$ which is defined in Subsection \ref{112}. To see this, notice that both the width of $R_k$ and the width of $\mathbb{G}a_k^{\rm u}(n_k;\ubar{z}^{(k)})$ tend to zero as $k$ goes to infinity, we need to prove the followings hold:
\begin{itemize}
\item[(i)] the width comparison (i.e. the ratio of widths of $R_k$ and $G^{\rm u}_k(n_k;\ubar{z}^{(k)})$) tends to zero as $k\to \infty$, and
\item[(ii)] the center $\bm{x}_k$ of $R_k$ is always located at the relative center position of $G^{\rm u}_k(n_k;\ubar{z}^{(k)})$ for every $k=1,2,\dots$, 
\end{itemize}
which together imply that, the left and right boundaries of $R_k$ do not exceed the boundaries of $\mathbb{G}a_k^{\rm u}(n_k;\ubar{z}^{(k)})$. In other words, the entire rectangle $R_k$ is wholly contained in $\mathbb{G}a_k^{\rm u}(n_k;\ubar{z}^{(k)})$.

Indeed, combining \eqref{075} and \eqref{029}, we have
\[\dfrac{\mbox{width}(R_k)}{|G^{\rm u}_k(n_k;\ubar{z}^{(k)})|}\lesssim \frac{|\bm{x}_k^l \bm{x}_k|_{\mathcal{F}^{\rm u}(\bm{x}_k)}}{\overline{\sigma}^{-n_k}}=\dfrac{1}{2}\rho\tilde{\beta}^{-1}\overline{\sigma}^{-n_k}\to 0\quad (k\to\infty),\] where the inequality comes from \eqref{029} and \eqref{030}.
This gives (i). For (ii), let us note that for $\varepsilon_0$ is defined in Subsection \ref{110}, the minimum distance between points on $\widetilde{L}$ and the vertical strips $S_{0,g}\cup S_{1,g}$ is greater than $(\frac{1}{2}-\sigma^{-1}-\varepsilon_0)$. In particular, the distance between $\bm{q}_k$ and the boundary of the center gap strip of $I^{\rm u}_g$ is bounded from below by this number. Here, we recall that 
$I^{\rm u}_g$ is the continuation of $I^{\rm u}$ for $g$  defined in Subsection \ref{071}. Thus, by the action of the backward iteration $g^{-n_k}$, recalling that $\bm{x}_k$ is the pre-image of $\bm{q}_k\in \widetilde{L}$ under $g^{n_k}$,  the distance of $\bm{x}_k$ and the the boundary of $\mathbb{G}a_k^{\rm u}(n_k;\ubar{z}^{(k)})$ is greater than 
$
\overline{\sigma}^{-n_k}(\frac{1}{2}-\sigma^{-1}-\varepsilon_0).
$
On the other hand, by \eqref{029}, the width $b_k$ of $R_k$ is no more than 
$
\rho\tilde{\beta}^{-1}\overline{\sigma}^{-2n_k}=O(\overline{\sigma}^{-2n_k}).
$
Hence (ii) holds for every sufficiently large $k$. 
 It follows that
\[R_k\subset \mathbb{G}a_k^{\rm u}(n_k;\ubar{z}^{(k)})\cap (Q \backslash (g(S_{0,g})\cup g(S_{1,g})))\] as desired in (2).

For (3), let $\ell ^{\rm u}_k$ be the segment of $\mathcal{F}^{\rm u}(\bm{x}_k)$ that connects $\bm{x}_k^l$ and $\bm{x}_k^r$ and $\ell^{\rm s}_k$ the segment of $\mathcal{F}^{\rm s}(\bm{x}_k)$ that connects $\bm{x}_k^t$ and $\bm{x}_k^b$.
We use 
\[\pi^\ast_{\mathcal{F}^{\rm u}}:R_k\to \ell^{\rm s}_k \quad\text{and}\quad \pi^\ast_{\mathcal{F}^{\rm s}}:R_k\to \ell^{\rm u}_k\]
to denote the projections along the leaves of $\mathcal{F}^{\rm u}$ and $\mathcal{F}^{\rm s}$ to $\ell^{\rm s}_k $ and $\ell^{\rm u}_k $ respectively.
First, Let us show that $g^{n_k+2}(\ell^{\rm u}_k)\subset \frac{1}{2}R_{k+1}$, where $\frac{1}{2}R_k$ is the rectangle defined in the same way as $R_k$ but replacing its width and height by half of those of $R_k$'s. See Figure \ref{fig9}. Indeed, on the one hand, we have
\begin{equation}\label{085}
\begin{split}
\big|\pi^\ast_{\mathcal{F}^{\rm u}}(g^{n_k+2}(\ell^{\rm u}_k))\big|_{\mathcal{F}^{\rm s}(\bm x_{k+1})}
&\lesssim \big|g^{n_k+2}(\ell_k^{\rm u})\big|
\lesssim \tilde{\alpha}\big|g^{n_k}(\ell_k^{\rm u})\big|\le \tilde{\alpha}\overline{\sigma}^{n_k}b_k\\
&=\rho\tilde{\alpha}\tilde{\beta}^{-1}\overline{\sigma}^{-\sum_{i=1}^\infty\frac{n_{k+i}}{2^i}}=\sqrt{\rho}\tilde{\alpha}\tilde{\beta}^{-\frac{1}{2}}\sqrt{b_{k+1}}.
\end{split}
\end{equation}   
Thus, by taking $\rho>0$ sufficiently small in \eqref{083}, we have
\begin{equation}\label{032}
\big|\pi^\ast_{\mathcal{F}^{\rm u}}(g^{n_k+2}(\ell^{\rm u}_k))\big|_{\mathcal{F}^{\rm s}(\bm x_{k+1})}<10\tilde{\alpha}\tilde{\beta}^{-\frac{1}{2}}\sqrt{b_{k+1}}
=\dfrac{1}{2}|\bm{x}_{k+1}^t \bm{x}^b_{k+1}|_{\mathcal{F}^{\rm s}(\bm{x}_{k+1})}.
\end{equation}
On the other hand, by \cite[Theorem 8 in Appendix 1]{PT93}, the curvature of the leaves $\mathcal{F}^{\rm s}(\bm{x})$ and $\mathcal{F}^{\rm u}(\bm{x})$ depend continuously on $\bm{x}$. Since the tangency between $g^{n_k+2}(\ell^{\rm u}_k)$ and $\ell^{\rm s}_{k+1}$ is quadratic, there is a constant $C$ independent of $k$, such that
\begin{equation*}
\begin{split}
\big|\pi^\ast_{\mathcal{F}^{\rm s}}(g^{n_k+2}(\ell^{\rm u}_k))\big|_{\mathcal{F}^{\rm u}(\bm x_{k+1})}
&\le C\big|\pi^\ast_{\mathcal{F}^{\rm u}}(g^{n_k+2}(\ell^{\rm u}_k))\big|_{\mathcal{F}^{\rm s}(\bm x_{k+1})}^2\\
&\lesssim (\rho\tilde{\alpha}\tilde{\beta}^{-1})^2\overline{\sigma}^{-\sum_{i=0}^\infty\frac{n_{k+1+i}}{2^{i}}},
\end{split}
\end{equation*}
where the last line follows from \eqref{085} and \eqref{083}. Notice that by shrinking $\rho$ if necessary,  
\[(\rho\tilde{\alpha}\tilde{\beta}^{-1})^2\overline{\sigma}^{-\sum_{i=0}^\infty\frac{n_{k+1+i}}{2^{i}}}=O(\rho^2)\]
can be made much smaller than 
\[\dfrac{1}{2}b_{k+1}=\dfrac{1}{2}\rho\tilde{\beta}^{-1}\overline{\sigma}^{-\sum_{i=0}^\infty\frac{n_{k+1+i}}{2^{i}}},\]
so that the following inequality 
\begin{equation}\label{033}
\big|\pi^\ast_{\mathcal{F}^{\rm s}}(g^{n_k+2}(\ell^{\rm u}_k))\big|_{\mathcal{F}^{\rm u}(\bm x_{k+1})}<\dfrac{1}{2}b_{k+1}=\dfrac{1}{2}|\bm{x}_{k+1}^l \bm{x}_{k+1}^r|_{\mathcal{F}^{\rm u}(\bm{x}_{k+1})}
\end{equation}
holds. Then \eqref{032} and \eqref{033} together imply that $g^{n_k+2}(\ell^{\rm u}_k)\subset \frac{1}{2}R_{k+1}$. 

Now, let us continue to show that $g^{n_k+2}(R_k)\subset R_{k+1}$. For every $\bm{x}\in \ell^{\rm u}_k$, write $\ell^{\rm s}_k(\bm{x}):=\mathcal{F}^{\rm s}_k(\bm{x})\cap R_k$. Thus, on the one hand, we have 
\begin{equation}\label{077}
\begin{split}
\big|g^{n_k+2}(\ell^{\rm s}_k(\bm{x}))\big|
&\lesssim \tilde{\gamma}\big|g^{n_k}(\ell^{\rm s}_k(\bm{x}))\big|\le \tilde{\gamma}\overline{\lambda}^{n_k}\big|\ell^{\rm s}_k(\bm{x})\big|\\
&\le 20 \tilde{\alpha} \tilde{\beta}^{-\frac{1}{2}} \tilde{\gamma}\overline{\lambda}^{n_k} \sqrt{b_k}
=20\tilde{\alpha}\tilde{\beta}^{-1}\tilde{\gamma}\sqrt{\rho}\overline{\lambda}^{n_k}\overline{\sigma}^{-\sum_{i=0}^\infty\frac{n_{k+i}}{2^{i+1}}}.
\end{split}
\end{equation}
On the other hand, \eqref{030} gives
\begin{equation}\label{078}
|\bm{x}_{k+1}^l \bm{x}_{k+1}|_{\mathcal{F}^{\rm u}(\bm{x}_{k+1})}=\dfrac{1}{2}b_{k+1}=\dfrac{1}{2}\rho\tilde{\beta}^{-1}\overline{\sigma}^{-\sum_{i=0}^\infty\frac{n_{k+1+i}}{2^{i}}}.
\end{equation}
Hence, by recalling \eqref{079}, it follows from \eqref{077} and \eqref{078} that we have the following width comparison: 
\begin{align*}
\dfrac{\big|g^{n_k+2}(\ell^{\rm s}_k(\bm{x}))\big|}{|\bm{x}_{k+1}^l \bm{x}_{k+1}|_{\mathcal{F}^{\rm u}(\bm{x}_{k+1})}}
&\lesssim \dfrac{40\tilde{\alpha}\tilde{\gamma}}{\sqrt{\rho}}\overline{\lambda}^{n_k}\overline{\sigma}^{\sum_{i=0}^\infty\frac{n_{k+1+i}}{2^i}-\sum_{i=0}^\infty\frac{n_{k+i}}{2^{i+1}} }\\
&\le \dfrac{40\tilde{\alpha}\tilde{\gamma}}{\sqrt{\rho}}\big(\overline{\lambda}\overline{\sigma}^{\frac{1+2\eta}{1-\eta}}\big)^{n_k}\to 0\quad (k\to \infty), 
\end{align*}
where the last inequality is obtained by a direct calculation together with \eqref{161}. 
Thus, when $\bm{x}$ travels along $\ell^{\rm u}_k$, we see that $g^{n_k+2}(\ell^{\rm s}_k(\bm{x}))$ can cover every point of $g^{n_k+2}(R_k)$.
Therefore, the width comparison and the fact that $g^{n_k+2}(\ell^{\rm u}_k)\subset \frac{1}{2}R_{k+1}$ together imply that $g^{n_k+2}(R_k)\subset R_{k+1}$ holds for every sufficiently large $k$. Finally, by translating the subscript (i.e. rename $R_k, R_{k+1}, R_{k+2}, \dots $ by $R_1, R_{2}, R_{3}, \dots $) if necessary, we proved (3), which also completes the proof of Lemma \ref{lem6}.
\end{proof}

\section{Proofs of main results}\label{sec8}
In the following proof, 
we identify certain combinatorial conditions, 
from which we derive the statistical conclusion of the first main result.

\begin{proof}[Proof of Theorem \ref{mainthm}]
Let $F_{0}$ be the diffeomorphism with the wild Smale horseshoe $\Lambda_{F_{0}}$ given in Section \ref{064}, 
and let $x$ be any element  of $\Lambda_{F_{0}}$.
Suppose that $\mathbb{U}_0$, $\mathbb{U}_1$ are small open regular neighborhoods of the 
rectangles $S_0$, $S_1$ in $M$ given in Subsection \ref{112} respectively.
For the coding map $h:\Lambda_{F_{0}}\longrightarrow \{0,1\}^{\mathbb{Z}}$ with respect to 
$\{\mathbb{U}_0, \mathbb{U}_1\}$, 
we set 
\begin{equation}\label{eqn_h(x)}
h(x)=\underline{v}=(\dots v_{-2}v_{-1}v_0v_1v_2\dots).
\end{equation}

For any $f\in \mathcal{U}_{0}$, 
we have 
a $C^{r}$ diffeomorphism $g\in \mathcal{U}_{0}$ with the following conditions:
$g$ is arbitrarily $C^r$-close to $f$ and
$g$ has a topological rectangle $R_{k}$ satisfying the conditions (1)--(3) of Lemma \ref{lem6}.
In particular, by Lemma \ref{lem6} (2), 
for any given  integer $k\geq 1$, $R_{k}$ is 
contained in the gap strip $\mathbb{G}a_k^{\rm u}(n_k;\ubar{z}^{(k)})$,   
where generation and itinerary are given by \eqref{020}  and \eqref{108} as 
\[
n_k=\widehat{u}_k+\widehat{m}_k+\widehat{s}_{k+1},\quad
\ubar{z}^{(k)}=\widehat{\ubar{z}}^{(k)}\widehat{\ubar{v}}^{(k)}[\widehat{\ubar{w}}^{(k+1)}]^{-1}.
\]

Here we consider the integer interval $\mathbb{I}_{k}=[\alpha_{k}, \alpha_{k}+\beta_{k}]\cap \mathbb{Z}$ with
\[
\alpha_{k}=\sum_{i=0}^{k-1} (n_{i}+2)+\widehat{u}_{k},\quad 
\beta_{k}=\widehat{m}_{k},
\]
and $n_{i}=\widehat{u}_{i}+\widehat{m}_{i}+\widehat{s}_{i+1}$ is the generation of 
the itinerary $\ubar{z}^{(i)}$ given in \eqref{108}. 
See Figure \ref{fig10}. 
\begin{figure}[hbt]
\centering
\scalebox{0.95}{\includegraphics[clip]{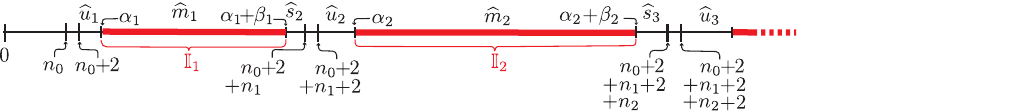}}
\caption{The integer intervals $\mathbb{I}_k$.} 
\label{fig10}
\end{figure}
Then one can take the middle part $\widehat{\ubar{v}}^{(k)}$ has the form 
\begin{equation}\label{eqn_hatvk}
\widehat{\ubar{v}}^{(k)}=(v_{\alpha_k+1}v_{\alpha_k+2}\dots  v_{\alpha_k+\beta_k}).
\end{equation}
For any $q\in \mathbb{N}$, we set $\mathbb{I}_k^{\,(q)}=[\alpha_{k}+q, \alpha_{k}+\beta_{k}-q]\cap \mathbb{Z}$ 
if $2q\leq \beta_k$ and otherwise $\mathbb{I}_k^{\,(q)}=\emptyset$.

For any integer $N\geq \alpha_1+\beta_1+1$, let $k_N$ be the greatest integer with 
$\alpha_{k_N}+\beta_{k_N}\leq N-1$.
It follows from Lemma \ref{lem5} (3) and \eqref{109} 
that, for any $\varepsilon >0$ and $q\in \mathbb{N}$, there exists an integer $N_0=N_0(\varepsilon,q)>0$ 
such that, for any $N\geq N_0$,   
\begin{align*}
\frac{\#\left\{ 0\le n \le N-1 \, :\,  n \in \bigcup _{k=1}^\infty \mathbb{I}_k^{\,(q)}\right\}}{N}
\geq 
\frac{\sum_{k=1}^{k_{N}}(\widehat{m}_{k}-2q)}{ \sum_{k=1}^{k_N+1}(\widehat{u}_{k}+\widehat{m}_{k}+\widehat{s}_{k+1}+2)}\\
=
\frac{\sum_{k=1}^{k_{N}}k^{2}-2qk_{N}}{ \sum_{k=1}^{k_N+1}(k^{2}+O(k))+2(k_N+1)}
=\frac{2k_N^{3}/6+O(k_N^{2})}{2k_N^{3}/6+O(k_N^{2})}>1-\varepsilon.
\end{align*}
This implies that  
\begin{equation}\label{eqn_Nepsilon}
\#\Bigl\{[0,N-1]\cap \mathbb{Z}\setminus \bigcup _{k=1}^\infty \mathbb{I}_k^{\,(q)}\Bigr\}<N\varepsilon
\quad\text{if}\quad N\geq N_0.
\end{equation}
We set $\mathrm{Int}(R_{1})=D$.
By Lemma \ref{lem6} (2) and \eqref{eqn_hatvk}, 
\begin{equation}\label{eqn_gnRk}
g^n(D)\subset g^{n}(R_{1})\subset \mathbb{U}_{v_{n}} 
\end{equation}
if $n\in \bigcup_{k\in \mathbb{N}}(\mathbb{I}_{k}\setminus \{\alpha_k\})$.

Since $g$ is sufficiently $C^{r}$ close to $f$ and hence to $F$ if $\mathcal{U}_{0}$ is sufficiently close to $F$ (see Remark \ref{rmk7.2}),
one can suppose that $\bigcap_{\,i\in \mathbb{Z}}g^i(\mathbb{U}_0\sqcup \mathbb{U}_1)$ 
is equal to the continuation $\Lambda_g$ of $\Lambda$.
Then there exists an integer $N_1>0$ such that, for any integer $k>0$ with $k^2>2N_1$ and 
any $j\in \mathbb{I}_k^{\,(N_1)}$, 
$$\mathrm{diam}\biggl(\,\bigcap_{i\in (\mathbb{I}_{k}\setminus \{\alpha_k\})}g^{j-i}(\mathbb{U}_{v_i})\biggr)\leq 
\mathrm{diam}\biggl(\,\bigcap_{u=-N_1}^{N_1}g^{-u}(\mathbb{U}_{v_{j+u}})\biggr)<\varepsilon.$$
By \eqref{eqn_gnRk}, $g^j(D)\subset \bigcap_{\,i\in (\mathbb{I}_{k}\setminus \{\alpha_k\})}g^{j-i}(\mathbb{U}_{v_i})$.
By \eqref{eqn_h(x)}, the continuation $x_g\in \Lambda_g$ of $x$ 
satisfies $\{x_g\}=\bigcap_{\,i\in \mathbb{Z}}g^{-i}(\mathbb{U}_{v_i})$ 
and hence $g^j(x_g)\in \bigcap_{\,i\in (\mathbb{I}_{k}\setminus \{\alpha_k\})}g^{j-i}(\mathbb{U}_{v_i})$.
Thus we have 
$$\sup_{y\in D}\mathrm{dist}(g^j(y),g^j(x_g))\leq \varepsilon$$ 
for any $j\in \mathbb{I}_k^{\,(N_1)}$.
By this fact together with \eqref{eqn_Nepsilon} for $q=N_1$,  
\begin{align*}
\sum_{j=0}^{N-1}
\sup_{y\in D}\mathrm{dist}(g^j(y),g^j(x_g))&=
\sum_{j\in \bigcup_{k=1}^\infty \mathbb{I}_k^{(N_1)}\cap [0,N-1]}\sup_{y\in D}\mathrm{dist}(g^j(y),g^j(x_g))\\
&\qquad\qquad+
\sum_{j\in [0,N-1]\cap\mathbb{Z}\setminus \bigcup_{k=1}^\infty\mathbb{I}_k^{(N_1)}}
\sup_{y\in D}\mathrm{dist}(g^j(y),g^j(x_g))\\
&< N\varepsilon+N\varepsilon\,\mathrm{diam}(M)=N\varepsilon(1+\mathrm{diam}(M))
\end{align*}
for any sufficiently large $N\in \mathbb{N}$.
Since one can take $\varepsilon$ arbitrarily small, the equation \eqref{defspl} holds.
This ends the proof of Theorem \ref{mainthm}.
\end{proof}

Next, 
Theorem \ref{HD} follows immediately from the next result.
\begin{prop}\label{thm0131}
Suppose that
$\mathcal{U}_0$ is the $C^{r}$-neighborhood  of $F_{0}$ with the wild horseshoe $\Lambda_{F_{0}}$ in Theorem \ref{mainthm}. 
Then, for any $f\in\mathcal{U}_0$, the following conditions hold.
\begin{enumerate}[\rm (1)]
\item \label{DrcHst-1} 
For every Birkhoff regular  $x\in \Lambda_{F_{0}}$ of $F_{0}$, 
there is a diffeomorphism $g\in \mathcal{U}_0$  
which is arbitrarily $C^{r}$-close to $f$ and  has 
a non-trivial physical measure supported on the forward $g$-orbit of the continuation $x_{g}\in \Lambda_{g}$ of $x$.
\item \label{DrcHst-2} 
There is a diffeomorphism $g\in \mathcal{U}_0$   
which is arbitrarily $C^{r}$-close to $f$ and 
 has a non-trivial contracting wandering domain $D$ such that 
the forward orbit of any point in $D$ has historic behavior.
\end{enumerate}
\end{prop}

\begin{proof}[Proof \rm (including the proof of Theorem  \ref{HD})]
First, we give the proof of  \eqref{DrcHst-1}.
Let $g$ be the diffeomorphism obtained in Lemma \ref{lem6} and
$\Lambda_{g}$  the wild horseshoe for $g$. 
The continuation 
 $x_{g}\in\Lambda_{g}$ is Birkhoff regular . 
 Let
\[h(x_{g})=(\dots v_{-2} v_{-1}v_{0} v_{1} v_{2} \dots) \in \{0,1\}^{\mathbb Z}\]
be the code of $x_{g}$, 
where $h:\Lambda_{g}  \longrightarrow  \{ 0,1\} ^{\mathbb Z}$ is the coding map 
given by $g^{i}(x_{g})\in \mathbb{U}_{v_{i}}$. 
Same as the proof of Theorem \ref{mainthm}, 
we here consider 
the itinerary $\ubar{z}^{(k)}=\widehat{\ubar{z}}^{(k)}\widehat{\ubar{v}}^{(k)}[\widehat{\ubar{w}}^{(k+1)}]^{-1}$ of the 
gap strip $\mathbb{G}a_k^{\rm u}(n_k;\ubar{z}^{(k)})$ containing 
$R_{k}$. 
Since one can choose any element of $\{0,1\}^{k^{2}}$ as the middle part $\widehat{\ubar{v}}^{(k)}$ of $\ubar{z}^{(k)}$, 
we assign the $0$th through $(k^{2}-1)$th entries of the above code of $h(x_{g})$  to $\widehat{\ubar{v}}^{(k)}$ as 
\[
\widehat{\ubar{v}}^{(k)}=(v_{0} v_{1} v_{2}\dots v_{k^{2}-1}).
\]
This implies that  $g$  
 has a non-trivial  physical measure supported on the forward orbit of $x_{g}$. 
 The remaining calculations are similar to those in the proof of \cite[Theorem 5.5]{KNS23}.
 This concludes the proof of \eqref{DrcHst-1}.
 \smallskip
 
 Next,  let us prove  \eqref{DrcHst-2}. To realize historic behavior in 
 the forward orbit starting from the contracting wandering domain $D=\mathrm{Int}(R_{1})$, 
 we prepare a code that oscillates between different dynamics in each generation and does not converge on any of them. 
 The easiest way might be the following.

\begin{itemize}
\item (Era condition)
We first consider 
an increasing sequence of integers 
 $(k_{s})_{s\in \mathbb{N}}$ 
such that, 
for every $s\in \mathbb{N}$,
\begin{equation}\label{era-ratio}
\sum_{k=k_{s}}^{k_{s+1}-1} k^{2}>s \sum_{k=1}^{k_{s}-1} k^{2}.
\end{equation}
\end{itemize}
Note that \eqref{era-ratio} provides the situation 
that the new era from 
$1$ to $k_{s+1}-1$ is so dominant that 
the old era from $1$ to $k_{s}-1$ is neglectable. 
\begin{itemize}
\item (Code condition for oscillation)
Under the condition \eqref{era-ratio}, for each integer $k\ge 1$, 
let 
$\ubar{v}^{(k)}=(v_{0} v_{1} v_{2}\dots v_{k^{2}-1})$
be the code whose 
entries satisfy the following rules: 
 \begin{subequations}
\begin{enumerate}[(1)]
\item 
if $s$ is even and $k_{s}\le k< k_{s+1}$, 
\[
v_{i}=\left\{
\begin{array}{ll}
0 & \text{for}\ i=0,\ldots, \left\lfloor k^{2}/3\right\rfloor-1
\\[3pt]
1 & \text{for}\ i=\left\lfloor k^{2}/3\right\rfloor,\ldots, k^2-1, 
\end{array}\right.
\]
that is, 
\[\widehat{\ubar{v}}^{(k)}=
\underbrace{0\ldots 0}_{\left\lfloor k^{2}/3\right\rfloor}
\underbrace{111\ldots\ldots 1}_{\left\lceil 2k^{2}/3\right\rceil},\]
\item
if $s$ is odd and $k_{s}\le k< k_{s+1}$,  
\[
v_{i}=\left\{
\begin{array}{ll}
0 & \text{for}\ i=0,\ldots, \left\lfloor 2k^{2}/3\right\rfloor-1
\\[3pt]
1 & \text{for}\  i=\left\lfloor 2k^{2}/3\right\rfloor,\ldots, k^{2}-1,
\end{array}\right.
\]
that is, 
\[\widehat{\ubar{v}}^{(k)}=
\underbrace{000\ldots\ldots 0}_{\left\lfloor 2k^{2}/3\right\rfloor}
\underbrace{1\ldots 1}_{\left\lceil k^{2}/3 \right\rceil},\]  
\end{enumerate}
\end{subequations}
where $\left\lfloor\cdot \right\rfloor$ and $\lceil\cdot\rceil$ indicate the 
floor and ceiling functions, respectively. 
\end{itemize}
The above ratio values such as 1/3 or 2/3 are not so essential,  
but the ratios should vary depending on whether the era is even or odd.

Using the above results, one can obtain a wandering domain $D$ with historic behavior.
In fact, consider the rectangle $R_{k}$ which is contained in $\mathbb{G}a_k^{\rm u}(n_k;\ubar{z}^{(k)})$, 
where $\ubar{z}^{(k)}=\widehat{\ubar{z}}^{(k)}\widehat{\ubar{v}}^{(k)}[\widehat{\ubar{w}}^{(k+1)}]^{-1}$
and the middle part $\widehat{\ubar{v}}^{(k)}$ satisfies the above code condition for oscillation.
This implies  that $D:=\mathrm{Int}(R_{1})$ 
is a wandering domain of $g$ whose forward orbit has historic behavior. 
The remaining calculations are the same as the proof of \cite[Theorem 5.1]{KNS23}. 
 This completes the proof of \eqref{DrcHst-2}.
 \end{proof}

\section*{Acknowledgements}
This work was partially supported by 
JSPS KAKENHI Grant Numbers 21K03332, 22K03342, 23K03188, 
Fapesp Grants 2022/07212-2,  2023/14277-6, 
NSFC Numbers 11701199, 12331005, and CSC 202206165004.
Li and Vargas acknowledge the warm hospitality of Tokai University (Japan), and Kiriki thanks  Universidade de S\~ao Paulo (Brazil) for their kindness. Finally, the authors thank the anonymous referees for their careful reading and helpful suggestions.

\bibliographystyle{amsalpha}
\bibliography{ref}

\end{document}